\documentclass[oneside,english]{siamart190516}
\usepackage[T1]{fontenc}
\usepackage[latin9]{inputenc}
\usepackage{float}
\usepackage{url}
\usepackage{amsmath}
\usepackage{amssymb}
\usepackage{graphicx}

\makeatletter

\providecommand{\tabularnewline}{\\}
\floatstyle{ruled}
\newfloat{algorithm}{tbp}{loa}
\providecommand{\algorithmname}{Algorithm}
\floatname{algorithm}{\protect\algorithmname}

\theoremstyle{definition}
\newtheorem{example}{\protect\examplename}
\theoremstyle{remark}
\newtheorem{claim}{\protect\claimname}
\theoremstyle{definition}
\newtheorem{defn}{\protect\definitionname}
\theoremstyle{plain}
\newtheorem{thm}{\protect\theoremname}
\ifx\proof\undefined\
  \newenvironment{proof}[1][\proofname]{\par
    \normalfont\topsep6\p@\@plus6\p@\relax
    \trivlist
    \itemindent\parindent
    \item[\hskip\labelsep
          \scshape
      #1]\ignorespaces
  }{%
    \endtrivlist\@endpefalse
  }
  \providecommand{\proofname}{Proof}
\fi
\theoremstyle{plain}
\newtheorem{cor}{\protect\corollaryname}
\theoremstyle{plain}
\newtheorem{lem}{\protect\lemmaname}
\theoremstyle{remark}
\newtheorem{rem}{\protect\remarkname}

\usepackage{algpseudocode}
\usepackage{hyperref}

\@ifundefined{showcaptionsetup}{}{%
 \PassOptionsToPackage{caption=false}{subfig}}
\usepackage{subfig}
\makeatother

\usepackage{babel}
\providecommand{\claimname}{Claim}
\providecommand{\corollaryname}{Corollary}
\providecommand{\definitionname}{Definition}
\providecommand{\examplename}{Example}
\providecommand{\lemmaname}{Lemma}
\providecommand{\remarkname}{Remark}
\providecommand{\theoremname}{Theorem}

\begin{document}
\title{Optimal Solutions of Well-Posed Linear Systems via\\
Low-Precision Right-Preconditioned GMRES with\\
Forward and Backward Stabilization}
\author{Xiangmin Jiao\footnotemark[1]\ \footnotemark[2]}

\maketitle
\footnotetext[1]{Department of Applied Mathematics \& Statistics and Institute for Advanced Computational Science, Stony Brook University, Stony Brook, NY 11794, USA.}
\footnotetext[2]{Email: xiangmin.jiao@stonybrook.edu.}
\headers{Optimal Solutions of Well-Posed Linear Systems}{X. Jiao}
\begin{abstract}
In scientific applications, linear systems are typically well-posed,
and yet the coefficient matrices may be nearly singular in that the
condition number $\kappa(\boldsymbol{A})$ may be close to $1/\varepsilon_{w}$,
where $\varepsilon_{w}$ denotes the unit roundoff of the working
precision. Accurate and efficient solutions to such systems pose daunting
challenges. It is well known that iterative refinement (IR) can make
the forward error independent of $\kappa(\boldsymbol{A})$ if $\kappa(\boldsymbol{A})$
is sufficiently smaller than $1/\varepsilon_{w}$ and the residual
is computed in higher precision. Recently, Carson and Higham {[}\emph{SISC},
39(6), 2017{]} proposed a variant of IR called GMRES-IR, which replaced
the triangular solves in IR with left-preconditioned GMRES using the
LU factorization of $\boldsymbol{A}$ as the preconditioner. GMRES-IR
relaxed the requirement on $\kappa(\boldsymbol{A})$ in IR, but it
requires triangular solves to be evaluated in higher precision, complicating
its application to large-scale sparse systems. We propose a new iterative
method, called \emph{Forward-and-Backward Stabilized Minimal Residual}
or \emph{FBSMR}, by conceptually hybridizing right-preconditioned
GMRES (RP-GMRES) with quasi-minimization. We develop FBSMR based on
a new theoretical framework of \emph{essential-forward-and-backward
stability} (\emph{EFBS}), which extends the backward error analysis
to consider the \emph{intrinsic condition number} of a well-posed
problem. We stabilize the forward and backward errors in RP-GMRES
to achieve EFBS by evaluating a small portion of the algorithm in
higher precision while evaluating the preconditioner in lower precision.
FBSMR can achieve optimal accuracy in terms of both forward and backward
errors for well-posed problems with unpolluted matrices, independently
of $\kappa(\boldsymbol{A})$. With low-precision preconditioning,
FBSMR can reduce the computational, memory, and energy requirements
over direct methods with or without IR. FBSMR can also leverage parallelization-friendly
classical Gram-Schmidt in Arnoldi iterations without compromising
EFBS. We demonstrate the effectiveness of FBSMR using both random
and realistic linear systems.
\end{abstract}

\section{Introduction}

We consider the accurate and efficient solutions of a mathematically
well-posed but potentially nearly singular linear system 
\begin{equation}
\boldsymbol{A}\boldsymbol{x}=\boldsymbol{b},\label{eq:linear-system}
\end{equation}
where $\boldsymbol{A}\in\mathbb{C}^{n\times n}$ is in general sparse
and large scale, $\boldsymbol{b}\in\mathcal{B}\subset\mathbb{C}^{n}$,
and $\boldsymbol{x}\in\mathbb{C}^{n}$. For ease of understanding,
$\boldsymbol{A}$, $\boldsymbol{b}$, and $\boldsymbol{x}$ may also
be assumed real in most of this work. By \textquotedblleft well-posedness,\textquotedblright{}
we mean that the solution $\boldsymbol{x}$ exists, is unique, and
continuously depends on $\boldsymbol{b}$ for a subset $\mathcal{B}\subset\mathbb{C}^{n}$
of practical interest (a.k.a., the well-posedness in the Hadamard
notion \cite{hadamard1902problemes}). In other words, the \emph{intrinsic}
\emph{condition number} of the problem for the specific $\boldsymbol{b}$,
denoted by $\kappa(\boldsymbol{A},\boldsymbol{b})$, is bounded by
a small constant, a notion also informally referred to as ``effective
well-conditioning'' \cite{chan1988effectively}. We emphasize that
well-posedness contains an implicit assumption that we do not care
whether the problem is stable for $\boldsymbol{b}\in\mathbb{C}^{n}\backslash\mathcal{B}$.
Despite well-posedness, $\boldsymbol{A}$ may be ill-conditioned in
that the condition number $\kappa(\boldsymbol{A})$ may be close to
$1/\varepsilon_{w}$, where $\varepsilon_{w}$ is the working precision
(or unit roundoff) of the floating-point representation of $\boldsymbol{A}$
or $\boldsymbol{b}$.

By ``accurate solution,'' we mean that both the forward and backward
errors are at the order of $\varepsilon_{w}$. More precisely, the
numerical solution $\hat{\boldsymbol{x}}$ satisfies that
\begin{equation}
\boldsymbol{A}\hat{\boldsymbol{x}}=\boldsymbol{b}+\boldsymbol{e},\label{eq:accuracy-forward-backward}
\end{equation}
where the forward and backward errors depend on $\kappa(\boldsymbol{A},\boldsymbol{b})$
independently of $\kappa(\boldsymbol{A})$, i.e., $\Vert\hat{\boldsymbol{x}}-\boldsymbol{x}_{*}\Vert=\mathcal{O}(\kappa(\boldsymbol{A},\boldsymbol{b})\varepsilon_{w})\Vert\boldsymbol{x}_{*}\Vert$
and $\Vert\boldsymbol{e}\Vert=\mathcal{O}(\varepsilon_{w})\Vert\boldsymbol{b}\Vert$.
The constant factors in the big-$\mathcal{O}$ notation may depend
on small low-degree polynomials in $n$ but not on numerical values
of $\boldsymbol{A}$ or $\boldsymbol{b}$, and $\Vert\cdot\Vert$
denotes the 2-norm throughout this work unless otherwise noted. We
refer to this strong notion of stability as \emph{essential-forward-and-backward
stability} (\emph{EFBS)}. EFBS is stronger than backward stability,
which is defined as
\begin{equation}
(\boldsymbol{A}+\boldsymbol{E})\hat{\boldsymbol{x}}=\boldsymbol{b}+\boldsymbol{e},\label{eq:standard-backward-stable}
\end{equation}
where $\vert\boldsymbol{E}\vert=\vert\boldsymbol{A}\vert\mathcal{O}(\varepsilon_{w})$
(in componentwise errors, or more weakly $\Vert\boldsymbol{E}\Vert=\Vert\boldsymbol{A}\Vert\mathcal{O}(\varepsilon_{w})$)
and $\Vert\boldsymbol{e}\Vert=\mathcal{O}(\varepsilon_{w})\Vert\boldsymbol{b}\Vert$.
EFBS is desirable if $\boldsymbol{A}$ is ``unpolluted'' in the
sense that $\Vert\delta\boldsymbol{A}\boldsymbol{x}\Vert=\mathcal{O}(\varepsilon_{w})\Vert\boldsymbol{b}\Vert$,
where $\delta\boldsymbol{A}$ denotes the rounding errors in $\boldsymbol{A}$
at input. We will give rigorous mathematical justification of this
condition in section~\ref{sec:efbs}.

Our ``efficiency'' goal is closely related to accuracy; otherwise,
one can always achieve accuracy by simply solving (\ref{eq:linear-system})
using exact or arbitrary precision throughout. More specifically,
we aim to solve the problem using an iterative method with an approximate-inverse
preconditioner, where the overwhelming amount of the computation is
performed using a lower precision of $\mathcal{O}(\sqrt{\varepsilon_{w}})$
(or even $\mathcal{O}(\sqrt[4]{\varepsilon_{w}})$), while the remaining
small portion of the computation is performed at the precision of
$\varepsilon_{w}$ or $\varepsilon_{w}/\kappa(\boldsymbol{A})$. Such
a solver is of practical interest since it can improve efficiency
in terms of computational cost, storage, and energy. It can also potentially
benefit from the increased ubiquity of half-precision floating-point
systems on computers ranging from laptops to exascale computers \cite{abdelfattah2020survey}.
To this end, we propose a method called \emph{Forward and Backward
Stabilized Minimum Residual} or \emph{FBSMR}. We develop FBSMR by
conceptually hybridizing right-preconditioned GMRES (RP-GMRES) with
quasi-minimization. In FBSMR, the preconditioner can be in lower precision,
and only some gaxpy operations (including matrix-vector multiplications
and the computation of the residual vector) are in higher precision.
For example, given $\boldsymbol{A}$ and $\boldsymbol{b}$ in double
precision, FBSMR can use an approximate inverse of $\boldsymbol{A}$
constructed and evaluated in single precision, and only evaluates
sparse gaxpy in double-double precision.\footnote{The double-double precision is a more portable and faster alternative
of quadruple precision \cite{briggs1998doubledouble}.} We defer the generalization to half-precision preconditioners to
future work, which requires extra care to avoid overflow and underflow.

This work can be considered as a continuation of \cite{jiao2022approximate},
in which we focused on the optimal preconditioning of a more general
case of (\ref{eq:linear-system}) where $\boldsymbol{A}$ may be singular
instead of being merely ill-conditioned. When $\boldsymbol{A}$ is
singular, well-posedness can be achieved with a more general formulation,
\begin{equation}
\boldsymbol{x}=\arg\min_{\boldsymbol{x}}\Vert\boldsymbol{x}\Vert\quad\text{subject to}\quad\min_{\boldsymbol{x}}\Vert\boldsymbol{b}-\boldsymbol{A}\boldsymbol{x}\Vert.\label{eq:pseudo-inverse-solution}
\end{equation}
Mathematically, the solutions to (\ref{eq:linear-system}) and (\ref{eq:pseudo-inverse-solution})
are $\boldsymbol{x}=\boldsymbol{A}^{-1}\boldsymbol{b}$ and $\boldsymbol{x}=\boldsymbol{A}^{+}\boldsymbol{b}$,
respectively, where $\boldsymbol{A}^{+}$ denotes the Moore-Penrose
pseudoinverse (see e.g., \cite{golub2013matrix}). The norms for the
residual $\boldsymbol{b}-\boldsymbol{A}\boldsymbol{x}$ and the solution
$\boldsymbol{x}$ in (\ref{eq:pseudo-inverse-solution}) may be replaced
by other weighted norms by scaling the rows and columns $\boldsymbol{A}$
correspondingly based on the physical meaning of the applications.
If (\ref{eq:pseudo-inverse-solution}) is inconsistent, treating a
singular system as a nonsingular system in a higher precision violates
the well-posedness and may lead to nonphysical solutions. To avoid
this difficulty, we focus on solving (\ref{eq:linear-system}) under
the assumption of nonsingular linear systems and solving (\ref{eq:pseudo-inverse-solution})
in a least-squares sense assuming consistency. We defer the further
improvement for the pseudoinverse solution of potentially inconsistent
least squares systems to future work.

The development of this work was motivated by a somewhat ``surprising''
finding in applying the preconditioner developed in \cite{jiao2022approximate}.
In particular, optimal preconditioners in \cite{jiao2022approximate}
were developed for FGMRES \cite{saad1993flexible}, a variant of RP-GMRES
with variable preconditioners. Most of the proofs in \cite{jiao2022approximate}
assumed exact arithmetic. In the presence of rounding errors, we referred
to the convergence theory of unpreconditioned GMRES \cite{Saad:2003aa}.
Although seemingly obvious in hindsight, we found that RP-GMRES (including
FGMRES) requires a different analysis from unpreconditioned GMRES
in terms of stability. More precisely, the estimated residual norm
in RP-GMRES can be severely distorted if $\kappa(\boldsymbol{A})$
is large, giving a false sense of accuracy in terms of the residual
or solution for a well-posed problem. For left-preconditioned GMRES
(LP-GMRES) with a preconditioner $\boldsymbol{M}$, a similar distortion
by $\kappa(\boldsymbol{M})$ is well understood (see e.g., \cite{ghai2019comparison}).
For RP-GMRES, this distortion is more subtle and it nevertheless prevents
an optimally preconditioned solver from achieving machine precision
for unpolluted well-posed problems. This finding was surprising in
that the conventional wisdom was that for iterative methods, ``a
sound strategy is to focus on finding a good preconditioner rather
than the best accelerator'' \cite[p. 254]{Saad:2003aa}. With optimal
preconditioners, however, the lack of guarantee of forward errors
in iterative methods has become \textsl{the} bottleneck in achieving
optimal accuracy. We developed the EFBS theory and the FBSMR algorithm
to overcome this issue. 

The objective of EFBS is related to \emph{extra-precision iterative
refinement} (\emph{EPIR}),\footnote{Most implementations of iterative refinement (IR) use the working
precision to evaluate the residual to achieve backward stability \cite{skeel1980iterative},
so we add extra-precision to the name as in \cite{demmel2006error}
to avoid confusion. Following \cite{demmel1997applied}, we refer
to the common practice as single-precision IR (SPIR).} which performs the factorization and triangular solves in the working
precision but evaluates the residual in higher precision \cite{moler1967iterative}.
If $\kappa(\boldsymbol{A})\varepsilon_{w}<1$, EPIR can make the forward
error independent of $\kappa(\boldsymbol{A})$ (see e.g. \cite[Theorem 2.7]{demmel1997applied}).
Hence, EPIR achieves our accuracy objectives in terms of the forward
error. EPIR does not ensure a small backward error, so its stopping
criteria must be based on the forward error \cite{demmel2006error}.
The consideration of backward errors in EFBS will simplify the stopping
criteria. Another difference between EPIR and this work is that EPIR
does not consider lower-precision factorization. To the best of our
knowledge, FBSMR is the first method to deliver an $\mathcal{O}(\varepsilon_{w})$
forward error for well-posed problems even when the factorization
and the triangular solves are both performed in lower precision. In
addition, since EPIR with higher-precision residuals is rarely implemented
for sparse solvers, FBSMR offers a valuable alternative to EPIR to
improve the forward errors for sparse systems, even when a working-precision
factorization is sufficiently efficient.

Our proposed FBSMR shares some similarities to those of \cite{arioli2009using}
and \cite{carson2017new,nicholas2019error}, which were designed as
improvements of SPIR and EPIR by leveraging a lower-precision factorization
of $\boldsymbol{A}$ as the preconditioner in FGMRES and LP-GMRES,
respectively. For the convenience of discussions, Algorithm~\ref{alg:PGMRES}
outlines a generic GMRES with two-sided preconditioners to cover both
\cite{arioli2009using} and \cite{carson2017new}, where $\boldsymbol{M}=\boldsymbol{M}_{\ell}\boldsymbol{M}_{r}=\boldsymbol{A}+\boldsymbol{E}$
for $\vert\boldsymbol{E}\vert=\mathcal{O}(\sqrt{\varepsilon_{w}})\vert\boldsymbol{A}\vert$.
In \cite[Theorem 3.1]{arioli2009using}, Arioli and Duff showed that
FGMRES (i.e., $\boldsymbol{M}_{\ell}=\boldsymbol{I}$ and restart=\textbf{false}
in Algorithm~\ref{alg:PGMRES}) using such an $\boldsymbol{M}$ can
achieve backward stability in the sense that 
\begin{equation}
\Vert\boldsymbol{b}-\boldsymbol{A}\hat{\boldsymbol{x}}\Vert\leq\mathcal{O}(\varepsilon_{w})(\Vert\boldsymbol{A}\Vert\Vert\hat{\boldsymbol{x}}\Vert+\Vert\boldsymbol{b}\Vert+\Vert\boldsymbol{A}\Vert\left\Vert \vert\boldsymbol{Z}_{k}\vert\,\vert\boldsymbol{y}\vert\right\Vert +\left\Vert \boldsymbol{A}\boldsymbol{Z}_{k}\right\Vert \Vert\boldsymbol{y}\Vert).\label{eq:fgmres-ad}
\end{equation}
This result motivated them to replace the standard convergence criterion
in GMRES (i.e., $\Vert\boldsymbol{b}-\boldsymbol{A}\hat{\boldsymbol{x}}\Vert\approx\Vert\boldsymbol{H}_{k}\boldsymbol{y}-\beta\boldsymbol{e}_{1}\Vert\leq\epsilon\Vert\boldsymbol{b}\Vert$)
\cite{Saad:2003aa} with $\Vert\boldsymbol{H}_{k}\boldsymbol{y}-\beta\boldsymbol{e}_{1}\Vert\leq\epsilon(\Vert\boldsymbol{A}\Vert\Vert\hat{\boldsymbol{x}}\Vert+\Vert\boldsymbol{b}\Vert)$,
where $\epsilon>\varepsilon_{w}$ is a user-controllable threshold.
This condition is equivalent to setting $\tau=\epsilon(\Vert\boldsymbol{A}\Vert\Vert\hat{\boldsymbol{x}}\Vert+\Vert\boldsymbol{b}\Vert)$
in Algorithm~\ref{alg:PGMRES}. To achieve a stronger sense of backward
stability, Arioli and Duff further added a safeguard to check $\Vert\boldsymbol{b}-\boldsymbol{A}\hat{\boldsymbol{x}}\Vert\leq\epsilon(\Vert\boldsymbol{A}\Vert\Vert\hat{\boldsymbol{x}}\Vert+\Vert\boldsymbol{b}\Vert)$
and restart GMRES if the condition is violated. However, they did
not provide a bound on the forward error. In addition, it is difficult
to estimate $\Vert\boldsymbol{A}\Vert$ (in 2-norm) for sparse matrices,
since most sparse solvers only estimate $\Vert\boldsymbol{A}\Vert_{1}$
or $\Vert\boldsymbol{A}\Vert_{\infty}$. In \cite{carson2017new,nicholas2019error},
Carson and Higham proposed a method called GMRES-IR ($\boldsymbol{M}_{r}=\boldsymbol{I}$,
$\boldsymbol{Z}_{k}=\boldsymbol{Q}_{k-1}$, and restart=\textbf{false}
in Algorithm~\ref{alg:PGMRES}), which replaced triangular solves
in IR using LP-GMRES, where all the operations in line~\ref{ln:prec-matvec}
are evaluated in the $\varepsilon_{w}^{2}$ precision. The stopping
criteria in GMRES-IR rely on the forward-error-based criteria of IR
\cite{demmel2006error}. For the residual-based stopping criteria
in LP-GMRES, GMRES-IR used $\tau=\epsilon\Vert\boldsymbol{b}\Vert$
with a fairly large $\epsilon$ (such as $10^{-4}$). For dense matrices,
assuming GMRES-IR requires $o(n)$ LP-GMRES iterations, GMRES-IR (as
described in \cite{nicholas2019error}) meets our efficiency objective,
because factorization requires $\mathcal{O}(n^{3})$ operations while
triangular solvers require $\mathcal{O}(n^{2})$ operations. For sparse
systems, which is our focus, the triangular solves are often as expensive
as factorization, especially if the factorization is reused in a time-dependent
or nonlinear problem. Hence, it is desirable to evaluate the preconditioner
in lower precision. From a practical point of view, it is also quite
invasive to implement higher-precision triangular solves in a well-crafted
sparse solver (such as MUMPS \cite{amestoy2001mumps,amestoy2022combining}).
These complications make GMRES-IR beyond the reach of average users.
It is also worth noting that both \cite{arioli2009using} and \cite{carson2017new}
used modified Gram-Schmidt (MGS) in Arnoldi iterations to take advantage
of its backward stability \cite{paige2006modified}. We show that
this backward stability is neither sufficient nor necessary to achieve
EFBS for RP-GMRES. Instead, by stabilizing the gaxpy operations, FBSMR
remains EFBS even with classical Gram-Schmidt (CGS), making it more
friendly for parallel implementations.

\begin{algorithm}
\caption{\label{alg:PGMRES}GMRES with approximate inverse as left or right
preconditioner}

\begin{algorithmic}[1]

\State $\hat{\boldsymbol{x}}\leftarrow\boldsymbol{x}_{0};$ $\boldsymbol{r}\leftarrow\boldsymbol{b}-\boldsymbol{A}\boldsymbol{x}_{0}$
\Comment{$\boldsymbol{x}_{0}\equiv\boldsymbol{M}_{r}^{-1}\boldsymbol{b}$
in \cite{arioli2009using} and $\boldsymbol{x}_{0}\equiv\boldsymbol{0}$
in \cite{carson2017new}}

\State $\hat{\boldsymbol{r}}\leftarrow\boldsymbol{M}_{\ell}^{-1}\boldsymbol{r}$;
$\beta\leftarrow\Vert\hat{\boldsymbol{r}}\Vert$; $\boldsymbol{Q}_{0}\leftarrow[\hat{\boldsymbol{r}}/\beta]$;
$\hat{\boldsymbol{Q}}_{0}\leftarrow[\,]$; $\hat{\boldsymbol{R}}_{0}=[\,]$;
$k\leftarrow0$

\For{ $\text{it}=1,2,\dots,\text{maxit}$}

\State \label{ln:prec-matvec}$k\leftarrow k+1$; $\boldsymbol{z}_{k}\leftarrow\boldsymbol{M}_{r}^{-1}\boldsymbol{q}_{k-1}$;
$\boldsymbol{w}\leftarrow\boldsymbol{M}_{\ell}^{-1}\boldsymbol{A}\boldsymbol{z}_{k}$\Comment{Use
higher precision in \cite{carson2017new}}

\State \label{ln:Q}$[\boldsymbol{Q}_{k},\boldsymbol{h}_{k}]\leftarrow\text{updateArnoldiMGS}(\boldsymbol{Q}_{k-1},\boldsymbol{w})$
 \Comment{Arnoldi via MGS}

\State\label{ln:givens}$[\hat{\boldsymbol{Q}}_{k},\hat{\boldsymbol{R}}_{k}]\leftarrow\text{updateQR}(\hat{\boldsymbol{Q}}_{k-1},\hat{\boldsymbol{R}}_{k-1},\boldsymbol{h}_{k})$\Comment{QR
of $[\boldsymbol{h}_{1},\dots,\boldsymbol{h}_{k}]$ via Givens rots.}

\If{\label{ln:approx-conv-criteria}$\left|\boldsymbol{e}_{k+1}^{T}\hat{\boldsymbol{Q}}_{k}^{H}(\beta\boldsymbol{e}_{1})\right|\leq\tau$}
\Comment{Approximate checking of convergence}

\State\label{ln:compute-x}$\boldsymbol{y}_{k}\leftarrow\hat{\boldsymbol{R}}_{1:k,1:k,}^{-1}\hat{\boldsymbol{Q}}_{k}^{H}(\beta\boldsymbol{e}_{1})$

\State$\hat{\boldsymbol{x}}\leftarrow\hat{\boldsymbol{x}}+\boldsymbol{Z}_{k}\boldsymbol{y}$\Comment{$\boldsymbol{Z}_{k}\equiv[\boldsymbol{z}_{1},\dots,\boldsymbol{z}_{k}]$
in \cite{arioli2009using} and $\boldsymbol{Z}_{k}\equiv\boldsymbol{Q}_{k-1}$
in \cite{carson2017new}}

\If{\label{ln:safeguard} restart \textbf{and} $(\sigma\leftarrow\Vert\boldsymbol{r}\leftarrow\boldsymbol{b}-\boldsymbol{A}\hat{\boldsymbol{x}}\Vert)\geq\tau$}
\Comment{Adaptive restart in \cite{arioli2009using}}

\State $\hat{\boldsymbol{r}}\leftarrow\boldsymbol{M}_{\ell}^{-1}\boldsymbol{r}$;
$\beta\leftarrow\Vert\hat{\boldsymbol{r}}\Vert$; $\boldsymbol{Q}_{0}\leftarrow[\hat{\boldsymbol{r}}/\beta]$;
$\hat{\boldsymbol{Q}}_{0}\leftarrow[\,]$; $\hat{\boldsymbol{R}}_{0}=[\,]$;
$k\leftarrow0$

\Else

\State \textbf{break}  \Comment{Check convergence in EPIR in \cite{carson2017new}}

\EndIf\label{ln:conv-check-end}

\EndIf

\EndFor

\end{algorithmic}
\end{algorithm}

The remainder of the paper is organized as follows. In Section~\ref{sec:efbs},
we formally define EFBS and show its connection with Hadamard well-posedness
of the underlying problem. In Section~\ref{sec:forward-and-backward-errors},
we present a new analysis of the forward and backward errors in RP-GMRES.
In Section~\ref{sec:fbsmr}, we introduce FBSMR by conceptually hybridizing
RP-GMRES with quasi-minimization. In Section~\ref{sec:numerical-experiments},
we present some numerical results of FBSMR in solving both random
and ``real-life'' linear systems. Section~\ref{sec:Conclusions}
concludes the paper with a discussion on future research directions.

\section{\label{sec:efbs}Essential Forward and Backward Stability}

As alluded to in the introduction, this work focuses on developing
algorithms that are EFBS for problems that are well-posed and unpolluted.
Since EFBS is stronger than the standard backward stability, we offer
a mathematical justification for its significance for practical applications.
We then make an important connection between the intrinsic condition
number of a linear system with the well-posedness of its underlying
problem, especially when the rounding errors in the coefficient matrices
are strongly or weakly correlated. The notion of EFBS will allow us
to develop a new stability analysis of RP-GMRES and identify the key
components that need to be stabilized in RP-GMRES in later sections.

\subsection{\label{subsec:Justification-of-EFBS}Justification of EFBS}

We start with a justification of the notion of EFBS for unpolluted
problems, by revisiting one of the most fundamental works in numerical
analysis: In the first rigorous analysis of rounding errors of linear
solvers, von Neumann and Goldstine \cite{von1947numerical} (abbreviated
as v.N.--G. below) analyzed the accuracy of solving symmetric positive
definite (SPD) linear systems arising from well-posed (or ``mathematically
stable'' \cite{von1947numerical}) problems. v.N.--G. considered
the perturbations in the input (including modeling errors, noise in
observational data, discretization errors, and rounding errors). They
asserted that ``those perturbations will cause a parameter to deviate
from its ideal value, but this deviation takes place only once, and
is then valid with a constant value throughout the entire problem.''
As we understand nowadays, random perturbations in $\boldsymbol{A}$
can be as damaging as the rounding errors in a backward-sable algorithm
\cite{golub2013matrix,higham2002accuracy}, so it is justified to
focus on backward stable algorithms \textsl{in general}. However,
the perturbations in $\boldsymbol{A}$ are not necessarily ``random,''
and the perturbed system may remain self-consistent (or ``valid''
in the words of v.N.--G.). We demonstrate it with a specific example
as follows.
\begin{example}
\label{exa:fdm}Consider the Poisson equation $\Delta u=f(u)$ on
a structured mesh over $\Omega=[0,1]\times[0,1]$ with Dirichlet boundary
conditions $f(u)=g(u)$ over $\partial\Omega$. When solving it using
a finite difference method with centered differences (see e.g., \cite{heath2018scientific}),
the matrix $\boldsymbol{A}$ is composed of $-4h^{-2}$ in the diagonals
and $h^{-2}$ for the off-diagonal nonzero entries. Since both $1$
and $4$ are powers of $2$, all the nonzero entries would have exactly
the same relative roundoff. This matrix is unpolluted despite the
rounding errors because we can scale both $\boldsymbol{A}$ and $\boldsymbol{b}$
by $h^{2}$ so that $\boldsymbol{A}$ is exact. As a result, we only
need to consider the ``intrinsic'' sensitivity of (\ref{eq:linear-system})
relative to perturbations in $\boldsymbol{b}$ (including the effect
of rounding errors in $h^{2}$ as well as the rounding errors in evaluating
$f$ and $g$). If $f$ and $g$ are continuous, the well-posedness
of the Poisson equation in infinite dimensions guarantees that the
problem is intrinsically insensitive to small perturbations for the
specific $\boldsymbol{b}$ on a sufficiently fine mesh.
\end{example}
\textsl{\emph{This example may be overly simple, but it nevertheless
raises an important point.}}
\begin{claim}
\label{claim:The-rounding-errors}The rounding errors in the input
$\boldsymbol{A}$ \textsl{may} be strongly correlated so that the
perturbed system may remain self-consistent. \textsl{\emph{Hence,}}
a numerical algorithm should not assume that $\kappa(\boldsymbol{A})$
is \textsl{always} the intrinsic condition number for the perturbations
in $\boldsymbol{A}$, so backward stability is not always the correct
goal.
\end{claim}
Of course, the preceding claim by no means diminishes the usefulness
of backward-stable algorithms: If the perturbations in $\boldsymbol{A}$
are truly random, then ``we cannot justifiably criticize an algorithm
for returning an inaccurate $\hat{\boldsymbol{x}}$ if $\boldsymbol{A}$
is ill-conditioned relatively to the unit roundoff'' in the sense
of (\ref{eq:standard-backward-stable}) \cite[p. 102]{golub2013matrix}.
However, if a modeler (the person who constructs (\ref{eq:linear-system}))
can ensure $\boldsymbol{A}$ is unpolluted in the sense that its rounding
errors are highly correlated as in the above example (also see section~\ref{subsec:Pollution-errors}),
then paraphrasing Golub and Van Loan, ``we have every `right' to
pursue the development of a linear equation solver that renders the
exact solution to a nearby problem'' in the sense of (\ref{eq:accuracy-forward-backward})
instead of (\ref{eq:standard-backward-stable}), so that the errors
introduced by the algorithm do not dominate those caused by the rounding
errors in the input.

\subsection{\label{subsec:intrinsic-condition-numbers}Intrinsic condition numbers
of well-posed problems and inverse}

We now formalize the notion of the intrinsic condition number for
a well-posed problem in a Hadamard notion \cite{hadamard1902problemes}.
Although Hadamard only considered initial value problems, it is common
to use a generalization of his notion of well-posedness.
\begin{defn}
\label{def:Hadamard-well-posedness}A problem $\boldsymbol{y}=\boldsymbol{f}(\boldsymbol{x}):\mathcal{X}\rightarrow\mathcal{Y}$
is \emph{Hadamard well-posed} if the solution $\boldsymbol{y}$ exists,
is unique, and depends continuously on $\boldsymbol{x}$.
\end{defn}
Under the assumption of well-posedness, both (\ref{eq:linear-system})
and (\ref{eq:pseudo-inverse-solution}) are bijections, of which the
inverse problem is $\boldsymbol{A}\boldsymbol{x}$, which maps from
$\boldsymbol{x}\in\mathcal{R}(\boldsymbol{A}^{H})$ to $\boldsymbol{b}\in\mathcal{R}(\boldsymbol{A})$.
Hence, we limit our attention to bijections when defining backward
errors.
\begin{defn}
\label{def:forward-backward-error}Given a problem $\boldsymbol{y}=\boldsymbol{f}(\boldsymbol{x}):\mathbb{C}^{n}\rightarrow\mathbb{C}^{m}$
that is a bijection in a neighborhood of $\boldsymbol{x}$ and $\boldsymbol{y}$
in $\mathcal{X}\subseteq\mathbb{C}^{n}$ and $\mathcal{Y}\subseteq\mathbb{C}^{m}$,
so that the inverse $\boldsymbol{x}=\boldsymbol{f}^{-1}(\boldsymbol{y}):\mathcal{Y}\subseteq\mathbb{C}^{m}\rightarrow\mathcal{X}\subseteq\mathbb{C}^{n}$
is well defined, let $\tilde{\boldsymbol{y}}=\tilde{\boldsymbol{f}}(\tilde{\boldsymbol{x}})$
denote a perturbed well-posed problem with some perturbations to $\tilde{\boldsymbol{x}}$
and the parameters in $\tilde{\boldsymbol{f}}$ with a corresponding
inverse $\tilde{\boldsymbol{x}}=\tilde{\boldsymbol{f}}^{-1}(\tilde{\boldsymbol{y}})$
in the neighborhood of $\tilde{\boldsymbol{x}}$ and $\tilde{\boldsymbol{y}}$.
The (relative) \emph{forward error} is $\Vert\tilde{\boldsymbol{f}}(\tilde{\boldsymbol{x}})-\boldsymbol{f}(\boldsymbol{x})\Vert/\Vert\boldsymbol{f}(\boldsymbol{x})\Vert$.
The (\emph{relative}) \emph{backward error }is the forward error of
the inverse problem, i.e., $\left\Vert \boldsymbol{f}^{-1}\left(\tilde{\boldsymbol{y}}\right)-\boldsymbol{f}^{-1}(\boldsymbol{y})\right\Vert /\Vert\boldsymbol{f}^{-1}(\boldsymbol{y})\Vert=\left\Vert \boldsymbol{f}^{-1}\left(\tilde{\boldsymbol{f}}(\tilde{\boldsymbol{x}})\right)-\boldsymbol{x}\right\Vert /\Vert\boldsymbol{x}\Vert$.
\end{defn}
We will consider the ``pollution'' $\delta\boldsymbol{f}=\tilde{\boldsymbol{f}}-\boldsymbol{f}$
in section~\ref{subsec:Pollution-errors}. For the case of $\tilde{\boldsymbol{f}}=\boldsymbol{f}$,
the backward error in Definition~\ref{def:forward-backward-error}
reduces to its standard notion, and the condition number is the amplification
factor between the forward and backward errors.
\begin{defn}
\label{def:condition-number} The\emph{ (intrinsic) condition number}
(\emph{ICN}) of $\boldsymbol{f}$ in Definition~\ref{def:forward-backward-error}
w.r.t. to the perturbation $\delta\boldsymbol{x}\equiv\tilde{\boldsymbol{x}}-\boldsymbol{x}$
is the supremum of the ratio between the forward and backward errors,
i.e., 
\begin{align*}
\kappa_{\delta\boldsymbol{x}}(\boldsymbol{f},\boldsymbol{x}) & =\lim_{\epsilon\rightarrow0}\sup_{\Vert\delta\boldsymbol{x}\Vert=\epsilon}\left(\left.\frac{\Vert\boldsymbol{f}(\tilde{\boldsymbol{x}})-\boldsymbol{f}(\boldsymbol{x})\Vert}{\Vert\boldsymbol{f}(\boldsymbol{x})\Vert}\right/\frac{\Vert\tilde{\boldsymbol{x}}-\boldsymbol{x}\Vert}{\Vert\boldsymbol{x}\Vert}\right).
\end{align*}
The problem $\boldsymbol{f}(\boldsymbol{x})$ is \emph{intrinsically
well-conditioned} (\emph{IWC}) for a specific $\boldsymbol{x}$ if
$\kappa_{\delta\boldsymbol{x}}\leq C$ for some constant $C>0$.
\end{defn}
In Definition~\ref{def:condition-number}, $\kappa_{\delta\boldsymbol{x}}$
is the standard condition number in numerical analysis (see e.g.,
\cite{heath2018scientific,trefethen1997numerical}). Mathematically,
the continuity in Definition~\ref{def:Hadamard-well-posedness} requires
that $\kappa_{\delta\boldsymbol{x}}(\boldsymbol{f},\boldsymbol{x})\leq C$
for some constant $C>0$, so Hadamard well-posedness implies\emph{
}IWC if $\boldsymbol{f}$ is a bijection. Note that $\kappa_{\delta\boldsymbol{x}}$
would be distorted if $\boldsymbol{x}$ and $\boldsymbol{y}=\boldsymbol{f}(\boldsymbol{x})$
are scaled by arbitrary matrices $\boldsymbol{W}_{1}$ and $\boldsymbol{W}_{2}$
so that $\Vert\boldsymbol{W}_{1}\boldsymbol{x}\Vert/\Vert\boldsymbol{x}\Vert\gg\Vert\boldsymbol{W}_{2}\boldsymbol{y}\Vert/\Vert\boldsymbol{y}\Vert$
or vice versa. By assuming well-posedness, our definitions allow such
scalings if they make ill-posed problems well-posed, while prohibiting
scalings that make well-posted problems ill-posed.

When applying Definition~\ref{def:condition-number} to (\ref{eq:linear-system}),
the forward problem is $\boldsymbol{A}^{-1}\boldsymbol{b}$. The inverse
problem is $\boldsymbol{A}\boldsymbol{x}$, so the backward error
is 
\[
\Vert\boldsymbol{A}(\boldsymbol{A}+\delta\boldsymbol{A})^{-1}(\boldsymbol{b}+\delta\boldsymbol{b})-\boldsymbol{b}\Vert/\Vert\boldsymbol{b}\Vert=\Vert(\boldsymbol{I}+\delta\boldsymbol{A}\boldsymbol{A}^{-1})^{-1}(\boldsymbol{b}+\delta\boldsymbol{b})-\boldsymbol{b}\Vert/\Vert\boldsymbol{b}\Vert.
\]
Note that computing the coefficients of $\boldsymbol{A}$ cannot be
the inverse problem since $\boldsymbol{A}$ has a high $n^{2}$ dimensionality
and hence cannot be determined from $n$ values in $\boldsymbol{x}$.
Hence, the ICN w.r.t. the perturbation $\delta\boldsymbol{b}$ (assuming
$\delta\boldsymbol{A}=\boldsymbol{0}$) is
\begin{equation}
\kappa_{\delta\boldsymbol{b}}(\boldsymbol{A},\boldsymbol{b})=\lim_{\epsilon\rightarrow0}\sup_{\Vert\delta\boldsymbol{b}\Vert=\epsilon}\left(\left.\frac{\Vert\boldsymbol{A}^{-1}\delta\boldsymbol{b}\Vert}{\Vert\boldsymbol{x}\Vert}\right/\frac{\Vert\delta\boldsymbol{b}\Vert}{\Vert\boldsymbol{b}\Vert}\right),\label{eq:cond-number}
\end{equation}
where $\boldsymbol{x}=\boldsymbol{A}^{-1}\boldsymbol{b}$. The following
theorem gives the formula for computing $\kappa_{\delta\boldsymbol{b}}(\boldsymbol{A},\boldsymbol{b})$,
which can also be used as an alternative definition of $\kappa_{\delta\boldsymbol{b}}$
in place of (\ref{eq:cond-number}).
\begin{thm}
\label{thm:condition-number-instance}Consider the singular value
decomposition (SVD) of a nonsingular $\boldsymbol{A}=\boldsymbol{U}\boldsymbol{\Sigma}\boldsymbol{V}^{H}$
and a vector \textup{$\boldsymbol{b}=\Vert\boldsymbol{b}\Vert(c_{1}\boldsymbol{u}_{1}+c_{2}\boldsymbol{u}_{2}+\cdots+c_{n}\boldsymbol{u}_{n})\in\mathbb{C}^{n}\backslash\{\boldsymbol{0}\}$,
where $c_{i}\in\mathbb{C}$ and $\boldsymbol{u}_{i}$ is the $i$th
left singular vector corresponding to singular value $\sigma_{i}$.
Then,
\begin{equation}
\kappa_{\delta\boldsymbol{b}}(\boldsymbol{A},\boldsymbol{b})=1/\sqrt{{\textstyle \sum_{i=1}^{n}}\vert c_{i}\sigma_{n}/\sigma_{i}\vert^{2}}.\label{eq:tight-condition-number}
\end{equation}
}
\end{thm}
\begin{proof}
Observe that $\lim_{\epsilon\rightarrow0}\sup_{\Vert\delta\boldsymbol{b}\Vert=\epsilon}\Vert\boldsymbol{A}^{-1}\delta\boldsymbol{b}\Vert/\Vert\delta\boldsymbol{b}\Vert=\Vert\boldsymbol{A}^{-1}\Vert=1/\sigma_{n}$.
Since $\boldsymbol{x}=\boldsymbol{A}^{-1}\boldsymbol{b}=\Vert\boldsymbol{b}\Vert\sum_{i}c_{i}\boldsymbol{v}_{i}/\sigma_{i}$,
$\Vert\boldsymbol{x}\Vert=\Vert\boldsymbol{b}\Vert\sqrt{\sum_{i=1}^{n}\vert c_{i}\vert^{2}/\sigma_{i}^{2}}$.
Substituting these into (\ref{eq:cond-number}), we obtain $\kappa_{\delta\boldsymbol{b}}(\boldsymbol{A},\boldsymbol{b})=\Vert\boldsymbol{A}^{-1}\Vert/\sqrt{\sum_{i}\vert c_{i}\vert^{2}/\sigma_{i}^{2}},$
which simplifies to (\ref{eq:tight-condition-number}).
\end{proof}
Obviously, $\kappa_{\delta\boldsymbol{b}}\leq\sigma_{1}/\sigma_{n}=\kappa(\boldsymbol{A})$.
However, $\kappa_{\delta\boldsymbol{b}}$ depends on how $\boldsymbol{b}$
is distributed in terms of the singular vectors. With a similar argument,
we obtain the following condition number for its inverse problem.
\begin{cor}
\label{cor:cond-mat-vec}For the matrix-vector multiplication $\boldsymbol{A}\boldsymbol{x}$
with \textup{$\boldsymbol{x}=\Vert\boldsymbol{x}\Vert(d_{1}\boldsymbol{u}_{1}+d_{2}\boldsymbol{u}_{2}+\cdots+d_{n}\boldsymbol{u}_{n})\in\mathbb{C}^{n}\backslash\{\boldsymbol{0}\}$,
$\kappa_{\delta\boldsymbol{x}}(\boldsymbol{A},\boldsymbol{x})=1/\sqrt{\sum_{i=1}^{n}\vert d_{i}\sigma_{i}/\sigma_{1}\vert^{2}}$.}
\end{cor}
However, in the context of (\ref{eq:cond-number}), it is more useful
to express $\kappa_{\delta\boldsymbol{x}}$ in terms of the components
in $\boldsymbol{b}$ instead of $\boldsymbol{x}$.
\begin{thm}
\label{thm:cond-inv}For the matrix-vector multiplication $\boldsymbol{A}\boldsymbol{x}$
with $\boldsymbol{x}=\boldsymbol{A}^{-1}\boldsymbol{b}$, under the
same conditions as in Theorem~\ref{thm:condition-number-instance},
\textup{
\begin{equation}
\kappa_{\delta\boldsymbol{x}}(\boldsymbol{A},\boldsymbol{x})=\sqrt{{\textstyle \sum_{i=1}^{n}}\vert c_{i}\sigma_{1}/\sigma_{i}\vert^{2}}.\label{eq:tight-cond-inverse-problem}
\end{equation}
}
\end{thm}
\begin{proof}
By definition,
\[
\kappa_{\delta\boldsymbol{x}}(\boldsymbol{A},\boldsymbol{x})=\lim_{\epsilon\rightarrow0}\sup_{\Vert\delta\boldsymbol{b}\Vert=\epsilon}\left(\left.\frac{\Vert\boldsymbol{A}\delta\boldsymbol{x}\Vert}{\Vert\boldsymbol{b}\Vert}\right/\frac{\Vert\delta\boldsymbol{x}\Vert}{\Vert\boldsymbol{x}\Vert}\right).
\]
Hence, $\kappa_{\delta\boldsymbol{x}}(\boldsymbol{A},\boldsymbol{x})=\Vert\boldsymbol{A}\Vert\sqrt{\sum_{i=1}^{n}\vert c_{i}\vert^{2}/\sigma_{i}^{2}}=\sqrt{\sum_{i=1}^{n}\vert c_{i}\sigma_{1}/\sigma_{i}\vert^{2}}$.
\end{proof}
From Theorems~\ref{thm:condition-number-instance} and \ref{thm:cond-inv},
we obtain the following result, which can be considered as a \emph{conservation
law of the condition number} $\kappa(\boldsymbol{A})$.
\begin{thm}
\label{thm:conservation-condition-number}Given $\boldsymbol{x}=\boldsymbol{A}^{-1}\boldsymbol{b}$
as in Theorems~\ref{thm:condition-number-instance}, $\kappa_{\delta\boldsymbol{b}}(\boldsymbol{A},\boldsymbol{b})\kappa_{\delta\boldsymbol{x}}(\boldsymbol{A},\boldsymbol{x})=\kappa(\boldsymbol{A})$.
\end{thm}
\begin{proof}
$\kappa_{\delta\boldsymbol{b}}\kappa_{\delta\boldsymbol{x}}=\sqrt{\sum_{i=1}^{n}\vert c_{i}\sigma_{1}/\sigma_{i}\vert^{2}}/\sqrt{\sum_{i=1}^{n}\vert c_{i}\sigma_{n}/\sigma_{i}\vert^{2}}=\sigma_{1}/\sigma_{n}=\kappa(\boldsymbol{A})$.
\end{proof}
To solidify the understanding of Theorems~\ref{thm:condition-number-instance}--\ref{thm:conservation-condition-number},
let us consider two examples.
\begin{example}
\label{exa:well-conditioned} If a substantial component of $\boldsymbol{b}$
falls within the space corresponding to the smallest singular values,
i.e., there exist $C_{1}\geq1$ and $0<C_{2}$ independently of $\kappa(\boldsymbol{A})$
such that $\sum_{\sigma_{i}/\sigma_{n}\leq C_{1}}\vert c_{i}\vert^{2}\geq C_{2}$,
and then 
\[
\sum_{i=1}^{n}\vert c_{i}\sigma_{n}/\sigma_{i}\vert^{2}\geq\sum_{\sigma_{i}/\sigma_{n}\leq C_{1}}\vert c_{i}\sigma_{n}/\sigma_{i}\vert^{2}\geq\sum_{\sigma_{i}/\sigma_{n}\leq C_{1}}\vert c_{i}/C_{1}\vert^{2}\geq C_{2}/C_{1}^{2},
\]
and then $\kappa_{\delta\boldsymbol{b}}\leq C_{1}/\sqrt{C_{2}}$ and
$\kappa_{\delta\boldsymbol{x}}\geq\sqrt{C_{2}}/C_{1}\kappa(\boldsymbol{A})$.
\end{example}
The ``effective well conditioned'' problems in \cite[Theorem 1]{chan1988effectively}
(i.e., $\boldsymbol{b}$ falls within a subspace corresponding to
small singular values) are special cases of Example~\ref{exa:well-conditioned}
in terms of $\kappa_{\delta\boldsymbol{b}}\leq C_{1}/\sqrt{C_{2}}$. 

The following example analyzes random matrices generated using a generalization
of a technique known as ``randsvd'' (see e.g., \cite[Section 28.3]{higham2002accuracy}
and \cite[eq. (4.1)]{arioli2009using}).
\begin{example}
\label{exa:random}Suppose $c_{i}$ is uniform (i.e., $c_{i}=1/\sqrt{n}$
and $\boldsymbol{b}\approx\sum_{i}\boldsymbol{u}_{i}/\sqrt{n}$) and
the singular values of $\boldsymbol{A}$ are $\sigma_{i}=10^{-\alpha(\frac{i-1}{n-1})^{r}}$.
Then, $\sigma_{1}=1$, $\alpha_{n}=10^{-\alpha}$, and $\kappa(\boldsymbol{A})=10^{\alpha}$.
If $r=1$, then 
\begin{align*}
\sum_{i=1}^{n}\vert c_{i}\sigma_{n}/\sigma_{i}\vert^{2} & =\frac{1}{n}\sum_{i=1}^{n}10^{-2\alpha(\frac{n-i}{n-1})}\\
 & =\frac{1}{n}\sum_{j=0}^{n-1}10^{-2\alpha(\frac{j}{n-1})}\\
 & =\frac{1}{n}\frac{1-10^{-2\alpha(\frac{n}{n-1})}}{1-10^{\frac{2\alpha}{n-1}}}\\
 & \approx\frac{n-1}{2\alpha\ln10n},
\end{align*}
for $1\ll\alpha<\log(\vert\varepsilon_{w}\vert)\ll n/2$, where the
approximate equality is due to the Taylor series of $10^{\frac{2\alpha}{n-1}}$.
Hence, $\kappa_{\delta\boldsymbol{b}}=\mathcal{O}(\log(\kappa(\boldsymbol{A}))$
for $r=1$ and large $n$. Since $\kappa_{\delta\boldsymbol{b}}$
decreases as $r$ decreases, $\kappa_{\delta\boldsymbol{b}}$ is $o(\log(\kappa(\boldsymbol{A}))$
for $r<1$ and it converges to $\mathcal{O}(1)$ as $r\rightarrow0$.
If $r>1$, $\kappa_{\delta\boldsymbol{b}}$ is unbounded by a constant
as $n$ increases, and $\kappa_{\delta\boldsymbol{b}}$ converges
to $\mathcal{O}(\kappa(\boldsymbol{A}))$ as $r\rightarrow\infty$
and $n\rightarrow\infty$. For moderately sized $n$, $\kappa_{\delta\boldsymbol{b}}\leq\sqrt{n}$,
which is significantly smaller than $\kappa(\boldsymbol{A})$ if $\alpha\gg\log_{10}\sqrt{n}$.
In contrast, $\kappa_{\delta\boldsymbol{x}}$ is close to $\kappa(\boldsymbol{A})$
if $r\leq1$ or $\alpha\gg\log_{10}\sqrt{n}$. Hence, randsvd matrices
typically have a small $\kappa_{\delta\boldsymbol{b}}$ but a large
$\kappa_{\delta\boldsymbol{x}}$ (close to $10^{\alpha}$) if $\boldsymbol{b}$
is a random vector with a uniform distribution.
\end{example}
From Theorems~\ref{thm:condition-number-instance}--\ref{thm:conservation-condition-number}
and the above examples, we can see that a large $\kappa(\boldsymbol{A})$
\textsl{always} leads to overestimation of $\kappa_{\delta\boldsymbol{b}}$,
$\kappa_{\delta\boldsymbol{x}}$, or both. For linear systems arising
from well-posed (and even random) problems, $\kappa(\boldsymbol{A})$
always over-estimates $\kappa_{\delta\boldsymbol{b}}$, and the computation
of the residual vector $\boldsymbol{r}=\boldsymbol{b}-\boldsymbol{A}\tilde{\boldsymbol{x}}$
is intrinsically sensitive to the rounding errors in $\delta\boldsymbol{x}=\tilde{\boldsymbol{x}}-\boldsymbol{x}$
and will be as large as $\kappa(\boldsymbol{A})\Vert\delta\boldsymbol{x}\Vert$.
This phenomenon is particularly important in iterative methods that
use the residual as the stopping criteria (as they almost always do
except for EPIR), for which it is essential to consider the stability
of both the forward and backward errors.

\subsection{\label{subsec:Pollution-errors}Pollution in well-posed problems}

Definition~\ref{def:forward-backward-error} allows the consideration
of pollution $\delta\boldsymbol{f}=\tilde{\boldsymbol{f}}-\boldsymbol{f}$.
We now derive some measures for this error and its impact in the context
of (\ref{eq:linear-system}), for which $\delta\boldsymbol{f}\equiv\delta\boldsymbol{A}$.
 Let $\Delta\boldsymbol{x}$ denote the error in the solution of
(\ref{eq:linear-system}) due to both $\delta\boldsymbol{b}$ and
$\delta\boldsymbol{A}$, i.e., 
\begin{equation}
(\boldsymbol{A}+\delta\boldsymbol{A})(\boldsymbol{x}+\Delta\boldsymbol{x})=\boldsymbol{b}+\delta\boldsymbol{b}.\label{eq:total-perturb}
\end{equation}
The following two lemmas allow us to analyze $\Delta\boldsymbol{x}$
based on $\kappa_{\delta\boldsymbol{b}}$.
\begin{lem}
\label{lem:total-perturb}The total error $\Delta\boldsymbol{x}$
in (\ref{eq:total-perturb}) satisfies 
\begin{equation}
\boldsymbol{A}(\boldsymbol{x}+\Delta\boldsymbol{x})=\boldsymbol{b}+\sum_{i=0}^{\infty}\left(-\delta\boldsymbol{A}\boldsymbol{A}^{-1}\right)^{i}\left(\delta\boldsymbol{b}-\delta\boldsymbol{A}\boldsymbol{x}\right).\label{eq:total-perturb2}
\end{equation}
\end{lem}
\begin{proof}
Since $(\boldsymbol{A}+\delta\boldsymbol{A})\Delta\boldsymbol{x}=\boldsymbol{b}+\delta\boldsymbol{b}-(\boldsymbol{A}+\delta\boldsymbol{A})\boldsymbol{x}=\delta\boldsymbol{b}-\delta\boldsymbol{A}\boldsymbol{x}$,
we obtain $\Delta\boldsymbol{x}=(\boldsymbol{A}+\delta\boldsymbol{A})^{-1}\left(\delta\boldsymbol{b}-\delta\boldsymbol{A}\boldsymbol{x}\right)$
and 
\[
\boldsymbol{A}\Delta\boldsymbol{x}=\boldsymbol{A}(\boldsymbol{A}+\delta\boldsymbol{A})^{-1}\left(\delta\boldsymbol{b}-\delta\boldsymbol{A}\boldsymbol{x}\right)=(\boldsymbol{I}+\delta\boldsymbol{A}\boldsymbol{A}^{-1})^{-1}\left(\delta\boldsymbol{b}-\delta\boldsymbol{A}\boldsymbol{x}\right).
\]
Applying the Taylor series expansion of $(1+x)^{-1}$ or using telescoping,
we obtain
\begin{equation}
\boldsymbol{A}\Delta\boldsymbol{x}=\sum_{i=0}^{\infty}\left(-\delta\boldsymbol{A}\boldsymbol{A}^{-1}\right)^{i}\left(\delta\boldsymbol{b}-\delta\boldsymbol{A}\boldsymbol{x}\right).\label{eq:taylor-series}
\end{equation}
Adding $\boldsymbol{A}\boldsymbol{x}$ and $\boldsymbol{b}$ to the
two sides, respectively, we then obtain (\ref{eq:total-perturb2}).
\end{proof}
Lemma~\ref{lem:total-perturb} is useful because of the tight bound
on the right-hand side of (\ref{eq:total-perturb}) given by the following
lemma.
\begin{lem}
\label{lem:bound-perturb}Suppose $\kappa(\boldsymbol{A})\left\Vert \delta\boldsymbol{A}\right\Vert /\text{\ensuremath{\left\Vert \boldsymbol{A}\right\Vert }}<C$
for some constants $C<1$. Then, 
\begin{equation}
\left\Vert \sum_{i=0}^{\infty}\left(-\delta\boldsymbol{A}\boldsymbol{A}^{-1}\right)^{i}\left(\delta\boldsymbol{b}-\delta\boldsymbol{A}\boldsymbol{x}\right)\right\Vert \leq\frac{1}{1-C}\left\Vert \delta\boldsymbol{b}-\delta\boldsymbol{A}\boldsymbol{x}\right\Vert .\label{eq:bound-rhs}
\end{equation}
\end{lem}
\begin{proof}
Without loss of generality, assume $\Vert\boldsymbol{b}\Vert\neq0$
(since $\boldsymbol{b}=0$, in general, implies that $\delta\boldsymbol{b}=\boldsymbol{0}$
and $\boldsymbol{x}=\boldsymbol{0}$). Observe that
\begin{align*}
\left\Vert \sum_{i=0}^{\infty}\left(-\delta\boldsymbol{A}\boldsymbol{A}^{-1}\right)^{i}\left(\delta\boldsymbol{b}-\delta\boldsymbol{A}\boldsymbol{x}\right)\right\Vert  & \leq\sum_{i=0}^{\infty}\left\Vert \left(-\delta\boldsymbol{A}\boldsymbol{A}^{-1}\right)^{i}\left(\delta\boldsymbol{b}-\delta\boldsymbol{A}\boldsymbol{x}\right)\right\Vert \\
 & \leq\left\Vert \delta\boldsymbol{b}-\delta\boldsymbol{A}\boldsymbol{x}\right\Vert \sum_{i=0}^{\infty}\left\Vert \delta\boldsymbol{A}\boldsymbol{A}^{-1}\right\Vert ^{i-1}\\
 & =\frac{\left\Vert \delta\boldsymbol{b}-\delta\boldsymbol{A}\boldsymbol{x}\right\Vert }{1-\text{\ensuremath{\left\Vert \delta\boldsymbol{A}\boldsymbol{A}^{-1}\right\Vert }}}\left(1-\text{\ensuremath{\left\Vert \delta\boldsymbol{A}\boldsymbol{A}^{-1}\right\Vert }}^{\infty}\right)\\
 & \leq\frac{1}{1-C}\left\Vert \delta\boldsymbol{b}-\delta\boldsymbol{A}\boldsymbol{x}\right\Vert ,
\end{align*}
where the last inequality is due to $\text{\ensuremath{\left\Vert \delta\boldsymbol{A}\boldsymbol{A}^{-1}\right\Vert }}\leq\kappa(\boldsymbol{A})\left\Vert \delta\boldsymbol{A}\right\Vert /\text{\ensuremath{\left\Vert \boldsymbol{A}\right\Vert }}<C<1$.
\end{proof}
Lemma~\ref{lem:bound-perturb} suggests a measure of sensitivity
in terms of the pollution error $\delta\boldsymbol{A}$.
\begin{defn}
\label{def:ecn}The \emph{intrinsic condition number} (\emph{ICN})
of (\ref{eq:linear-system}) w.r.t. the pollution error $\delta\boldsymbol{A}$
is 
\begin{equation}
\hat{\kappa}_{\delta\boldsymbol{A}}(\boldsymbol{A},\boldsymbol{b})=\lim_{\epsilon\rightarrow0}\sup_{\Vert\delta\boldsymbol{A}\Vert=\epsilon}\left(\left.\frac{\Vert(\boldsymbol{A}+\delta\boldsymbol{A})^{-1}\boldsymbol{b}-\boldsymbol{x}\Vert}{\Vert\boldsymbol{x}\Vert}\right/\frac{\Vert\delta\boldsymbol{A}\boldsymbol{x}\Vert}{\Vert\boldsymbol{b}\Vert}\right).\label{eq:cond-deltaA-1}
\end{equation}
\end{defn}
We emphasize that ICN is not the standard condition number (e.g.,
as in \cite[(18.10)]{trefethen1997numerical}) since the backward
error is not measured as $\Vert\delta\boldsymbol{A}\Vert/\Vert\boldsymbol{A}\Vert$.
Hence, we put a hat on $\kappa$ to avoid confusion. $\hat{\kappa}_{\delta\boldsymbol{A}}$
is useful since $\delta\boldsymbol{A}\boldsymbol{x}$ defines a correlation
of the pollution $\delta\boldsymbol{A}$. The following theorem shows
that $\hat{\kappa}_{\delta\boldsymbol{A}}\leq\kappa_{\delta\boldsymbol{b}}\leq C$
for well-posed problems.
\begin{thm}
\label{thm:cond-num-relation}$\hat{\kappa}_{\delta\boldsymbol{A}}\leq\kappa_{\delta\boldsymbol{b}}$,
as defined in (\ref{eq:cond-deltaA-1}) and (\ref{eq:cond-number}),
respectively.
\end{thm}
\begin{proof}
Using a similar argument as in Lemma~\ref{lem:total-perturb}, we
obtain 
\[
(\boldsymbol{I}+\delta\boldsymbol{A}\boldsymbol{A}^{-1})^{-1}\boldsymbol{b}-\boldsymbol{b}=-\sum_{i=0}^{\infty}\left(-\delta\boldsymbol{A}\boldsymbol{A}^{-1}\right)^{i}\delta\boldsymbol{A}\boldsymbol{x}.
\]
For any $0<C<1$, $\kappa(\boldsymbol{A})\left\Vert \delta\boldsymbol{A}\right\Vert /\text{\ensuremath{\left\Vert \boldsymbol{A}\right\Vert }}<C$
for a sufficiently small $\varepsilon_{w}$. Using a similar argument
as in Lemma~\ref{lem:bound-perturb}, we obtain 
\[
\left\Vert (\boldsymbol{I}+\delta\boldsymbol{A}\boldsymbol{A}^{-1})^{-1}\boldsymbol{b}-\boldsymbol{b}\right\Vert \leq\frac{1}{1-C}\left\Vert \delta\boldsymbol{A}\boldsymbol{x}\right\Vert .
\]
Hence, the supremum in $\hat{\kappa}_{\delta\boldsymbol{A}}$ is no
greater than $(1-C)$ times of $\tilde{\kappa}_{\delta\boldsymbol{A}}$
defined as 
\begin{equation}
\tilde{\kappa}_{\delta\boldsymbol{A}}(\boldsymbol{A},\boldsymbol{b})=\lim_{\epsilon\rightarrow0}\sup_{\Vert\delta\boldsymbol{A}\Vert=\epsilon}\left(\left.\frac{\Vert(\boldsymbol{A}+\delta\boldsymbol{A})^{-1}\boldsymbol{b}-\boldsymbol{x}\Vert}{\Vert\boldsymbol{x}\Vert}\right/\frac{\Vert(\boldsymbol{I}+\delta\boldsymbol{A}\boldsymbol{A}^{-1})^{-1}\boldsymbol{b}-\boldsymbol{b}\Vert}{\Vert\boldsymbol{b}\Vert}\right).\label{eq:intemediate-cond}
\end{equation}
Taking the limit of $(1-C)$ as $C$ approaches $0$, $\hat{\kappa}_{\delta\boldsymbol{A}}\leq\text{\ensuremath{\tilde{\kappa}_{\delta\boldsymbol{A}}}}$.
Let $\delta\boldsymbol{b}=(\boldsymbol{I}+\delta\boldsymbol{A}\boldsymbol{A}^{-1})^{-1}\boldsymbol{b}-\boldsymbol{b}$
in (\ref{eq:cond-number}), which tends to $\boldsymbol{0}$ as $\Vert\delta\boldsymbol{A}\Vert$
approaches 0, so $\tilde{\kappa}_{\delta\boldsymbol{A}}\leq\kappa_{\delta\boldsymbol{b}}$.
\end{proof}
Theorem~\ref{thm:cond-num-relation} suggests a natural definition
for measuring the correlation of $\delta\boldsymbol{A}$.
\begin{defn}
\label{def:correlation}For a well-posed problem, the rounding errors
$\delta\boldsymbol{A}$ are \emph{strongly correlated} if $\Vert\delta\boldsymbol{A}\boldsymbol{x}\Vert/\Vert\boldsymbol{b}\Vert=\mathcal{O}(\varepsilon_{w})$,
\emph{weakly correlated} if $\Vert\delta\boldsymbol{A}\boldsymbol{x}\Vert/\Vert\boldsymbol{b}\Vert=o(\kappa(\boldsymbol{A})\varepsilon_{w})$,
and \emph{uncorrelated} if $\Vert\delta\boldsymbol{A}\boldsymbol{x}\Vert/\Vert\boldsymbol{b}\Vert=\Omega(\kappa(\boldsymbol{A})\varepsilon_{w})$.
\end{defn}
If the linear system is polluted with uncorrelated $\delta\boldsymbol{A}$,
then backward stability would suffice (within an $\mathcal{O}(1)$
factor). We will focus on cases where $\delta\boldsymbol{A}$ is strongly
correlated, which we informally refer to as \emph{unpolluted}. Under
this assumption, let $\hat{\boldsymbol{A}}=\boldsymbol{A}+\delta\boldsymbol{A}$
and $\hat{\boldsymbol{b}}=\boldsymbol{b}+\delta\boldsymbol{b}$ in
(\ref{eq:total-perturb}). If we solve (\ref{eq:total-perturb}) in
the sense of (\ref{eq:accuracy-forward-backward}), i.e., $\hat{\boldsymbol{A}}\hat{\boldsymbol{x}}=\hat{\boldsymbol{b}}+\boldsymbol{e}$,
where $\left\Vert \hat{\boldsymbol{x}}-(\boldsymbol{x}+\Delta\boldsymbol{x})\right\Vert /\left\Vert \boldsymbol{x}+\Delta\boldsymbol{x}\right\Vert =\mathcal{O}(\varepsilon_{w})$
and $\Vert\boldsymbol{e}\Vert/\Vert\hat{\boldsymbol{b}}\Vert=\mathcal{O}(\varepsilon_{w})$,
then
\[
\frac{\left\Vert \hat{\boldsymbol{x}}-\boldsymbol{x}\right\Vert }{\left\Vert \boldsymbol{x}\right\Vert }=\frac{\left\Vert \hat{\boldsymbol{x}}-\boldsymbol{x}\right\Vert }{\left\Vert \boldsymbol{x}+\Delta\boldsymbol{x}\right\Vert }(1+\mathcal{O}(\varepsilon_{w}))\leq\frac{\left\Vert \hat{\boldsymbol{x}}-(\boldsymbol{x}+\Delta\boldsymbol{x})\right\Vert +\left\Vert \Delta\boldsymbol{x}\right\Vert }{\left\Vert \boldsymbol{x}+\Delta\boldsymbol{x}\right\Vert }(1+\mathcal{O}(\varepsilon_{w}))=\mathcal{O}(\varepsilon_{w}),
\]
and similarly, 
\[
\frac{\left\Vert \hat{\boldsymbol{b}}+\boldsymbol{e}-\boldsymbol{b}\right\Vert }{\left\Vert \boldsymbol{b}\right\Vert }=\frac{\left\Vert \hat{\boldsymbol{b}}+\boldsymbol{e}-\boldsymbol{b}\right\Vert }{\Vert\hat{\boldsymbol{b}}\Vert}(1+\mathcal{O}(\varepsilon_{w}))\leq\frac{\left\Vert \boldsymbol{e}\right\Vert +\Vert\delta\boldsymbol{b}\Vert}{\Vert\hat{\boldsymbol{b}}\Vert}(1+\mathcal{O}(\varepsilon_{w}))=\mathcal{O}(\varepsilon_{w}).
\]
In other words, forward and backward errors are both guaranteed to
be small. Hence, in the remainder of the paper, we will assume that
$\hat{\boldsymbol{A}}=\boldsymbol{A}+\delta\boldsymbol{A}$ and $\hat{\boldsymbol{b}}=\boldsymbol{b}+\delta\boldsymbol{b}$
are the ``ground truth'' when solving (\ref{eq:linear-system})
(as v.N.-G. did in \cite{von1947numerical}), for which the exact
solution will be $\boldsymbol{x}+\Delta\boldsymbol{x}$. To avoid
clustering, we will omit the hats in $\hat{\boldsymbol{A}}$ and $\hat{\boldsymbol{b}}$.

\subsection{\label{subsec:essential-backward-stablity}Essential forward and
backward stability}

We now focus on solving (\ref{eq:linear-system}) with exact input
$\boldsymbol{A}$ and $\boldsymbol{b}$ with small forward and backward
errors.
\begin{defn}
\label{def:fs-ebs}Suppose $\boldsymbol{A}$ and $\boldsymbol{b}$
in (\ref{eq:linear-system}) are exactly represented by a floating-point
system with unit round-off $\varepsilon_{w}$, and let $\boldsymbol{x}_{*}$
denote the exact solution to (\ref{eq:linear-system}). Let $\hat{\boldsymbol{x}}$
denoted a numerical solution, which may be rounded from a higher-precision
solution $\tilde{\boldsymbol{x}}=\hat{\boldsymbol{x}}+\delta\boldsymbol{x}$
where $\vert\delta x_{i}|=\vert x_{i}|\mathcal{O}(\varepsilon_{w})$. 
\begin{itemize}
\item The method is \emph{essentially forward stable} (\emph{EFS}) if $\Vert\hat{\boldsymbol{x}}-\boldsymbol{x}_{*}\Vert\leq\Vert\boldsymbol{x}\Vert\mathcal{O}(\varepsilon_{w})$. 
\item The method is \emph{essentially backward stable} (\emph{EBS}) if the
residual in terms of $\tilde{\boldsymbol{x}}$ is $\mathcal{O}(\varepsilon_{w})$,
i.e., $\Vert\tilde{\boldsymbol{r}}\Vert\leq\Vert\boldsymbol{b}\Vert\mathcal{O}(\varepsilon_{w})$,
where $\tilde{\boldsymbol{r}}=\boldsymbol{b}-\boldsymbol{A}\tilde{\boldsymbol{x}}$.
This bound must be verifiable, even if $\tilde{\boldsymbol{x}}$ is
not formed explicitly.
\item The method is \emph{EFBS} if it is both EFS and EBS.
\end{itemize}
\end{defn}
Definition~\ref{def:fs-ebs} differs from the standard stability
and backward stability in numerical analysis in two senses. First,
the notion of EFS is stronger than stability \cite{golub2013matrix,trefethen1997numerical}
(i.e., $\Vert\tilde{\boldsymbol{x}}-\boldsymbol{x}_{*}\Vert\leq\Vert\boldsymbol{x}\Vert\kappa(\boldsymbol{A})\mathcal{O}(\varepsilon_{w})$)
or ``forward stability'' \cite[p. 9]{higham2002accuracy}. Furthermore,
we consider EFS as a goal separate from EBS, since an EFS solution
may be computed using some components that are not backward stable
or even unstable. Second, we define EBS based on a residual $\boldsymbol{b}-\boldsymbol{A}\tilde{\boldsymbol{x}}$
instead of $\boldsymbol{b}-\boldsymbol{A}\hat{\boldsymbol{x}}$, even
though $\hat{\boldsymbol{x}}$ is the final output. This detail is
important because $\Vert\boldsymbol{b}-\boldsymbol{A}\hat{\boldsymbol{x}}\Vert$
is close to $\mathcal{O}(\kappa(\boldsymbol{A})\varepsilon_{w})$
for well-posed problems, \textsl{even if $\boldsymbol{b}-\boldsymbol{A}\hat{\boldsymbol{x}}$
is computed in exact arithmetic}.

\section{\label{sec:forward-and-backward-errors}Forward and Backward Errors
in RP-GMRES}

We present a new analysis of forward and backward errors in RP-GMRES.
This analysis serves as the guideline in devising FBSMR, but the reader
can safely skip to section~\ref{sec:fbsmr} without compromising
the understanding of the implementation.

\subsection{Trial and test spaces in RP-GMRES}

As we have seen in Algorithm~\ref{alg:PGMRES}, RP-GMRES solves a
\emph{projected least squares }(\emph{PLS}) \emph{problem}, 
\begin{equation}
\boldsymbol{Q}_{k}^{H}\boldsymbol{A}\boldsymbol{Z}_{k}\boldsymbol{y}_{k}=\boldsymbol{Q}_{k}^{H}\boldsymbol{r}\label{eq:petrov-galerkin}
\end{equation}
where $\boldsymbol{Z}_{k}=\boldsymbol{M}^{-1}\boldsymbol{Q}{}_{k-1}$,
and then $\boldsymbol{x}_{k}=\boldsymbol{Z}_{k}\boldsymbol{y}_{k}$.
Using the analogy of variational methods for solving partial differential
equations (PDEs), we refer to the range spaces of $\boldsymbol{Z}_{k}$
and $\boldsymbol{Q}_{k}$ as the \emph{trial} and \emph{test spaces},
respectively. Since $\mathcal{R}(\boldsymbol{Z}_{k})$ and $\mathcal{R}(\boldsymbol{Q}_{k})$
differ substantially in general when $\boldsymbol{M}\neq\boldsymbol{I}$,
we consider RP-GMRES as a Petrov-Galerkin (PG) method.\footnote{Petrov-Galerkin methods are generalizations of Galerkin (aka Ritz-Galerkin)
methods for PDEs in that the test spaces may differ from trial spaces.
CG is a Galerkin-projection method \cite[Section 4.1]{van2003iterative},
and so are multigrid methods for symmetric systems from PDEs (e.g.,
\cite[Section 13]{Saad:2003aa}). } In contrast, unpreconditioned GMRES solves a PLS problem
\begin{equation}
\boldsymbol{Q}_{k}^{H}\boldsymbol{A}\boldsymbol{Q}_{k-1}\boldsymbol{y}_{k}=\boldsymbol{Q}_{k}^{H}\boldsymbol{r},\label{eq:projected-least-squares}
\end{equation}
for which the trial and test spaces are $\mathcal{R}(\boldsymbol{Q}_{k-1})$
and $\mathcal{R}(\boldsymbol{Q}_{k})$, respectively. Note that $\mathcal{R}(\boldsymbol{Q}_{k-1})=\mathcal{R}(\boldsymbol{Q}_{k})$
at convergence (or breakdown) in exact arithmetic, so GMRES is a Galerkin
(or ``quasi-Galerkin'') method.\footnote{LP-GMRES may also be considered a PG method in terms of $\boldsymbol{A}\boldsymbol{x}=\boldsymbol{b}$
where the trial and test spaces are $\mathcal{R}(\boldsymbol{Q}_{k-1})$
and $\mathcal{R}(\boldsymbol{M}^{-H}\boldsymbol{Q}_{k})$, respectively,
although it is typically treated as a quasi-Galerkin method in terms
of $\boldsymbol{M}^{-1}\boldsymbol{A}\boldsymbol{x}=\boldsymbol{M}^{-1}\boldsymbol{b}$.} Due to this difference, RP-GMRES requires more sophisticated definitions
and analyses for its trial and test spaces compared to those in \cite{drkovsova1995numerical,paige2006modified,Saad:2003aa}.
To this end, let us first refine the definition of the Krylov space
by taking into account the orthogonalization in Arnoldi iterations.
\begin{defn}
\label{def:Arnoldi-Krylov}Given $\boldsymbol{A}\in\mathbb{C}^{n\times n}$
and $\boldsymbol{v}\in\mathbb{C}^{n}$, the \emph{Arnoldi-Krylov subspace}
(\emph{AKS}) is given by
\begin{equation}
\mathcal{K}_{k}(\boldsymbol{A},\boldsymbol{v})=\langle\boldsymbol{q}_{0},\boldsymbol{A}\boldsymbol{q}_{0},\dots,\boldsymbol{A}^{k-1}\boldsymbol{q}_{k-1}\rangle,\label{eq:Arnoldi-Krylov}
\end{equation}
where $\boldsymbol{q}_{0}=\boldsymbol{v}/\Vert\boldsymbol{v}\Vert$,
$\Vert\boldsymbol{q}_{i}\Vert=1$, and $\boldsymbol{q}_{i+1}\in\boldsymbol{P}_{i}^{\perp}\mathcal{K}_{i}(\boldsymbol{A},\boldsymbol{v})$
for $i\geq1$, where $\boldsymbol{P}_{i}^{\perp}$ is the orthogonal
projector $\boldsymbol{P}_{i}^{\perp}=\boldsymbol{I}-\boldsymbol{Q}_{i-1}\boldsymbol{Q}_{i-1}^{H}=\boldsymbol{I}-\sum_{j=0}^{i-1}\boldsymbol{q}_{j}\boldsymbol{q}_{j}^{H}$
and $\boldsymbol{Q}_{i-1}\equiv[\boldsymbol{q}_{0},\boldsymbol{q}_{1},\dots,\boldsymbol{q}_{i-1}]$.
We refer to $\{\boldsymbol{q}_{i}\mid0\leq i\leq k-1\}$ as the \emph{generating
vectors} of $\mathcal{K}_{k}$.
\end{defn}
RP-GMRES involves the following two subspaces.
\begin{defn}
\label{def:primal-dual}Given $\boldsymbol{A}\in\mathbb{C}^{n\times n}$,
a nonsingular right preconditioner $\boldsymbol{M}\in\mathbb{C}^{n\times n}$,
and $\boldsymbol{v}\in\mathbb{C}^{n}$, the \emph{primal Arnoldi-Krylov
subspace} (\emph{PAKS}) is
\begin{equation}
\mathcal{K}_{k}(\boldsymbol{A}\boldsymbol{M}^{-1},\boldsymbol{v})=\langle\boldsymbol{q}_{0},\boldsymbol{A}\boldsymbol{M}^{-1}\boldsymbol{q}_{0},\dots,\left(\boldsymbol{A}\boldsymbol{M}^{-1}\right)^{k}\boldsymbol{q}_{k-1}\rangle,\label{eq:Primary-Arnoldi-Krylov}
\end{equation}
as defined in (\ref{eq:Arnoldi-Krylov}). The \emph{dual Krylov subspace}
(\emph{DKS}) is 
\begin{align}
\mathcal{D}_{k}(\boldsymbol{A},\boldsymbol{M},\boldsymbol{v}) & =\boldsymbol{M}^{-1}\mathcal{K}_{k}(\boldsymbol{A}\boldsymbol{M}^{-1},\boldsymbol{v})\equiv\mathcal{K}_{k}(\boldsymbol{M}^{-1}\boldsymbol{A},\boldsymbol{M}^{-1}\boldsymbol{v}),\label{eq:dual-space}
\end{align}
where the $\boldsymbol{q}_{i}$ are the same as the generating vectors
of the primal subspace.
\end{defn}
Note that the generating vectors of $\mathcal{D}_{k}$ (i.e., $\{\boldsymbol{M}^{-1}\boldsymbol{q}_{i}\}$)
are nonorthogonal, and hence we can only refer to $\mathcal{D}_{k}$
as a Krylov subspace instead of AKS. 

\subsection{\label{subsec:Stability-of-residual}Stability of PLS in RP-GMRES}

The analogy of RP-GMRES with PG allows us to gain some insights in
terms of its stability. First, observe that the basis vectors in PG
for PDEs (such as some finite element methods \cite{brenner2008mathematical})
in general do not need orthogonal basis functions. Hence, we assert
that the orthogonality of $\boldsymbol{Q}_{k}$ plays a minor role
for the stability of RP-GMRES. This assertion seemingly contradicts
the conventional wisdom of GMRES \cite{arioli2009using,drkovsova1995numerical,paige2006modified},
but it should not be surprising since QMR \cite{freund1993transpose},
which is also a PG method \cite{van2003iterative}, does not use orthogonal
basis vectors either and often converges (although less robustly than
GMRES) in practice. The mathematical reason is that the loss of orthogonality
of $\boldsymbol{Q}_{k}$ would lead to a quasi-minimization (QM, instead
of $\ell_{2}$ minimization). As long as QM leads to a reduction of
the true residual, then RP-GMERS can continue making progress, as
we show in Theorem~\ref{thm:restarted-rpgmres} below.

Second, in PG the trial space is the most critical for the accuracy
of the solution (see e.g., \cite{conley2020hybrid}). This statement
partially explains the success of GMRES-IR using LP-GMRES, of which
the trial space is $\mathcal{R}(\boldsymbol{Q}_{k-1})$ with orthonormal
basis vectors. For RP-GMRES, we assert that the key to achieve EFBS
is to stabilize the projection of $\boldsymbol{A}$ onto $\mathcal{D}_{k}$.
However, this projection cannot be stabilized by making $\mathcal{D}_{k}$
an AKS, because Arnoldi orthogonalization would only help achieve
backward stability unless it is done in higher precision, which we
want to avoid. Instead, we will achieve EFBS by stabilizing the projection
based on the following theorem.
\begin{thm}
\label{thm:restarted-rpgmres}In RP-GMRES with restart, let $\boldsymbol{r}_{0,k}$
denote the residual vector at the beginning since the most recent
restart (or start) and $s_{k}$ denote the number of steps since the
most recent restart. Suppose PLS in (\ref{eq:petrov-galerkin}) solves
a quasi-minimization
\begin{equation}
\boldsymbol{y}_{k}=\arg\min_{\boldsymbol{y}}\Vert\boldsymbol{r}_{0,k}-\boldsymbol{A}\boldsymbol{Z}_{k}\boldsymbol{y}\Vert_{\boldsymbol{W}}\label{eq:leastsq-alt}
\end{equation}
for some weighted semi-norm $\Vert\boldsymbol{v}\Vert_{\boldsymbol{W}}\equiv\Vert\boldsymbol{W}^{H}\boldsymbol{x}\Vert$,
such that $\delta\boldsymbol{x}_{k}=\boldsymbol{Z}_{k}\boldsymbol{y}_{k}$
reduces $\Vert\boldsymbol{b}-\boldsymbol{A}\boldsymbol{x}_{k}\Vert$
(in exact arithmetic) by $\Vert\boldsymbol{b}-\boldsymbol{A}\boldsymbol{x}_{k}\Vert/\Vert\boldsymbol{b}-\boldsymbol{A}\boldsymbol{x}_{k-s_{k}}\Vert\leq C$
for some $C<1$ before restart. Then, the backward error $\Vert\boldsymbol{b}-\boldsymbol{A}\boldsymbol{x}_{k}\Vert/\Vert\boldsymbol{b}\Vert$
converges to $\varepsilon_{w}$ after $\log_{C}(\varepsilon_{w})-\log_{C}\left(\Vert\boldsymbol{b}\Vert/\Vert\boldsymbol{b}-\boldsymbol{A}\boldsymbol{x}_{0}\Vert\right)$
restarts. For well-posed problems, the forward error $\Vert\boldsymbol{x}-\boldsymbol{x}_{k}\Vert/\Vert\boldsymbol{x}\Vert$
also converges to $\mathcal{O}(\varepsilon_{w})$ in exact arithmetic.
\end{thm}
\begin{proof}
By assumption, after $m$ restarts, the residual $\Vert\boldsymbol{b}-\boldsymbol{A}\boldsymbol{x}_{k}\Vert\leq C^{m}\Vert\boldsymbol{b}-\boldsymbol{A}\boldsymbol{x}_{0}\Vert$.
Hence, reducing the residual to $\varepsilon_{w}\Vert\boldsymbol{b}\Vert$
requires at most $\log_{C}(\varepsilon_{w}\Vert\boldsymbol{b}\Vert/\Vert\boldsymbol{b}-\boldsymbol{A}\boldsymbol{x}_{0}\Vert)=\log_{C}(\varepsilon_{w})-\log_{C}(\Vert\boldsymbol{b}\Vert/\Vert\boldsymbol{b}-\boldsymbol{A}\boldsymbol{x}_{0}\Vert)$
steps. Since $\Vert\boldsymbol{x}-\boldsymbol{x}_{k}\Vert/\Vert\boldsymbol{x}\Vert\leq\kappa(\boldsymbol{A},\boldsymbol{b})\Vert\boldsymbol{b}-\boldsymbol{A}\boldsymbol{x}_{k}\Vert/\Vert\boldsymbol{b}\Vert$,
where $\kappa(\boldsymbol{A},\boldsymbol{b})=\mathcal{O}(1)$ for
well-posed problems, the forward error (barring pollution in input
$\boldsymbol{A}$) also converges to $\mathcal{O}(\varepsilon_{w})$.
\end{proof}
In Theorem~\ref{thm:restarted-rpgmres}, QM in general does result
in a reduction in $\Vert\boldsymbol{r}_{0,k}-\boldsymbol{A}\boldsymbol{Z}_{k}\boldsymbol{y}\Vert$
at the $k$th step for small $k$ because $\Vert\boldsymbol{r}_{0,k}-\boldsymbol{A}\boldsymbol{Z}_{k}\boldsymbol{y}\Vert\lesssim\sqrt{k}\Vert\boldsymbol{r}_{0,k}-\boldsymbol{A}\boldsymbol{Z}_{k}\boldsymbol{y}\Vert_{\boldsymbol{W}}$
if $\boldsymbol{W}$ is composed of $k$ unit vectors. This argument
is also fundamental in QMR \cite{qmr}, especially for its convergence
criteria. Note that the $\sqrt{k}$ factor may be larger in the presence
of rounding errors, since $\boldsymbol{W}^{H}\boldsymbol{A}\boldsymbol{Z}_{k}$
is not a Hessenberg matrix anymore in general but we treat it as a
Hessenberg matrix (the same also holds in QMR, which treats a non-tridiagonal
matrix as tridiagonal in the presence of rounding errors). The orthogonality
of the vectors in $\boldsymbol{W}$ (i.e., $\boldsymbol{Q}_{k}$)
can reduce the $\sqrt{k}$ factor to a number closer to $1$, but
its effect is not as significant as the stability of the projection
of $\boldsymbol{A}$ onto $\mathcal{D}_{k}$. Hence, we expect that
it suffices to use the working precision when performing QM and to
use the unstable but more parallelization-friendly CGS in place of
MGS in Arnoldi orthogonalization, as we will demonstrate in section~\ref{subsec:Comparison-of-CGS}.
By assuming exact arithmetic when computing $\boldsymbol{b}-\boldsymbol{A}\boldsymbol{x}_{k}=\boldsymbol{r}_{0,k}-\boldsymbol{A}\boldsymbol{Z}_{k}\boldsymbol{y}$,
we ensure that the reduction in $\Vert\boldsymbol{r}_{0,k}-\boldsymbol{A}\boldsymbol{Z}_{k}\boldsymbol{y}\Vert$
translates to the reduction in $\Vert\boldsymbol{b}-\boldsymbol{A}\boldsymbol{x}_{k}\Vert$,
independently of $\kappa(\boldsymbol{A})$. In practice, ``exact
arithmetic'' can be replaced with $\varepsilon_{w}^{2}$ precision
(assuming $\kappa(\boldsymbol{A})\leq1/\varepsilon_{w}$) when computing
$\boldsymbol{A}\boldsymbol{Z}_{k}$, $\delta\boldsymbol{x}_{k}\equiv\boldsymbol{Z}_{k}\boldsymbol{y}_{k}$,
$\boldsymbol{x}_{k}\equiv\boldsymbol{x}_{k-s_{k}}+\delta\boldsymbol{x}_{k}$,
and $\boldsymbol{b}-\boldsymbol{A}\boldsymbol{x}_{k}$. This completes
the derivation of FBSMR.
\begin{rem}
Theorem~\ref{thm:restarted-rpgmres} shares some similarities in
spirit with \cite[Theorem 2.7]{demmel1997applied} and \cite[Theorem 2.1]{carson2017new},
which focus on forward errors in EPIR \cite{demmel2006error}. The
residual in EPIR is evaluated in higher precision, but its accuracy
is only guaranteed to about $\kappa(\boldsymbol{A})\varepsilon_{w}$
for well-posed problems. This inaccuracy prevented the use of residuals
as the convergence criteria in EPIR. By evaluating $\boldsymbol{x}_{k}$
in higher precision (at a minimal cost), we can compute the residual
to $\mathcal{O}(\varepsilon_{w})$, so that the residual can be used
as the convergence criteria in Theorem~\ref{thm:restarted-rpgmres}
and in FBSMR. Since both $\boldsymbol{x}_{k}$ and $\boldsymbol{r}_{0,k}$
are stabilized, we refer to our approach as \emph{forward and backward
stabilization} (\emph{FBS}).
\end{rem}

\section{\label{sec:fbsmr}Forward and Backward Stabilized Minimal Residual}

Based on the theory in sections~\ref{sec:efbs} and \ref{sec:forward-and-backward-errors},
we propose \emph{Forward-Backward-Stabilized Minimal Residual} or
\emph{FBSMR} as outlined in Algorithm~\ref{alg:FBSMR}. We consider
FBSMR as a conceptual (albeit not algorithmic, since we do not use
Lanczos biorthogonalization) hybridization of RP-GMRES with QMR since
the PLS in FBSMR is only a quasi-minimization. For this reason, we
use MR (instead of GMRES or QMR) as the base of its name.

\begin{algorithm}
\caption{\label{alg:FBSMR}FBSMR: Forward-Backward-Stabilized Minimal Residual}

\hspace*{\algorithmicindent} \textbf{Input}: $\boldsymbol{A}\in\mathbb{F}_{\varepsilon_{w}}^{n\times n}$;
$\boldsymbol{b}\in\mathbb{F}_{\varepsilon_{w}}^{n\times n}$; $\boldsymbol{M}\in\mathbb{F}_{\sqrt{\varepsilon_{w}}}^{n\times n}$;
$\tau$ ($10\varepsilon_{w}$); $\text{restart}$ ($30$); $\text{maxit}$
(500) \\
\hspace*{\algorithmicindent} \textbf{Output}: $\hat{\boldsymbol{x}}\in\mathbb{F}_{\varepsilon_{w}}^{n}$;
optionally $\tilde{\boldsymbol{x}}\in\mathbb{F}_{\varepsilon_{w}^{2}}^{n\times n}$
and $\gamma\equiv\Vert\boldsymbol{b}-\boldsymbol{A}\tilde{\boldsymbol{x}}\Vert/\Vert\boldsymbol{b}\Vert$

\begin{algorithmic}[1]

\State $\beta_{0}\leftarrow\Vert\boldsymbol{b}\Vert$; $\boldsymbol{x}_{0}\leftarrow\boldsymbol{M}^{-1}\boldsymbol{b}$;
$\boxed{\tilde{\boldsymbol{x}}\leftleftarrows\boldsymbol{x}_{0}}$;
$\boldsymbol{r}\leftarrow\boxed{\boldsymbol{b}-\boldsymbol{A}\tilde{\boldsymbol{x}}}$;
$\beta\leftarrow\Vert\boldsymbol{r}\Vert$; $\text{it}\leftarrow0$

\While{ $\text{it}<\text{maxit}$}

\State $\boldsymbol{Q}_{0}\leftarrow[\boldsymbol{r}/\beta]$; $\boldsymbol{g}\leftarrow\beta\boldsymbol{e}_{1}$

\For{ $k=1,\dots,\text{restart}$}

\State\label{ln:mat-vec-prod}$\text{it}\leftarrow\text{it}+1$;
$\boldsymbol{z}_{k}\leftarrow\boldsymbol{M}^{-1}\boldsymbol{q}_{k-1}$;
$\boldsymbol{w}\leftarrow\boxed{\boldsymbol{A}\boldsymbol{z}_{k}}$
\Comment{$\boldsymbol{Z}_{k}\equiv[\boldsymbol{z}_{1},\dots,\boldsymbol{z}_{k}]$}

\State\label{ln:Arnoldi}$\boldsymbol{h}\leftarrow\boldsymbol{Q}_{k-1}^{H}\boldsymbol{w}$;
$\boldsymbol{w}\leftarrow\boldsymbol{w}-\boldsymbol{Q}_{k-1}\boldsymbol{h}$;
$\alpha\leftarrow\Vert\boldsymbol{w}\Vert$; $\boldsymbol{q}_{k}\leftarrow\boldsymbol{w}/\alpha$

\State\label{ln:givens-1}$\hat{\boldsymbol{r}}_{1:k,k}\leftarrow\hat{\boldsymbol{\Omega}}_{k-1}\dots\hat{\boldsymbol{\Omega}}_{1}\boldsymbol{h}$;
$[\boldsymbol{\Omega}_{k},\hat{r}_{k,k}]\leftarrow\text{GR}\left(\hat{r}_{k,k},\alpha\right)$;
$\boldsymbol{g}_{k:k+1}\leftarrow\boldsymbol{\Omega}_{k}\boldsymbol{g}_{k:k+1}$

\State\label{ln:approx-conv}\textbf{break if }$\left|g_{k+1}\right|\leq\tau\beta_{0}$
\Comment{Approximate convergence check}

\EndFor

\State\label{ln:compute-solution-1}$\boldsymbol{y}_{k}\leftarrow\hat{\boldsymbol{R}}_{1:k,1:k,}^{-1}\boldsymbol{g}_{1:k}$;
$\boxed{\tilde{\boldsymbol{x}}\leftleftarrows\tilde{\boldsymbol{x}}+\boldsymbol{Z}_{k}\boldsymbol{y}_{k}}$;
$\boldsymbol{r}\leftarrow\boxed{\boldsymbol{b}-\boldsymbol{A}\tilde{\boldsymbol{x}}}$\Comment{Update
sol. and res.}

\State\label{ln:update-residual-1}$\beta\leftarrow\Vert\boldsymbol{r}\Vert$;
$\gamma\leftarrow\beta/\beta_{0}$ \Comment{Compute backward error}

\If{\label{ln:true-conv} $\gamma\le\tau$}

\State $\hat{\boldsymbol{x}}\leftarrow\tilde{\boldsymbol{x}}$; \textbf{break
}\Comment{Converged}

\Else

\State\textbf{continue} \Comment{Restart}

\EndIf

\EndWhile

\end{algorithmic}
\end{algorithm}

In FBSMR, the input $\boldsymbol{A}$ and $\boldsymbol{b}$ are floating-point
real or complex numbers in working precision. Typically, $\boldsymbol{A}$
should be unpolluted, although the algorithm itself does not depend
on this property. The approximate-inverse preconditioner $\boldsymbol{M}$
can be in a lower precision with unit roundoff $\sqrt{\varepsilon_{w}}$,
but it may also be in another precision (such as $\varepsilon_{w}$
or $\sqrt[4]{\varepsilon_{w}}$; we will pursue the latter in future
work). The algorithm has three control parameters: $\tau$ (the convergence
tolerance), $\text{restart}$ (the maximum number of iterations before
restart), and $\text{maxit}$ (the maximum number of iterations).
The recommended default values are $\tau=10\varepsilon_{w}$, restart=30,
and maxit=500. The algorithm returns an approximate solution $\hat{\boldsymbol{x}}$
in working precision. Optionally, FBSMR returns $\tilde{\boldsymbol{x}}$
in higher precision and/or the relative residual $\gamma$ computed
from $\tilde{\boldsymbol{x}}$. In general, $\tilde{\boldsymbol{x}}$
is no more accurate than $\hat{\boldsymbol{x}}$, but it is required
to evaluate the residual accurately for convergence criteria (line~\ref{ln:true-conv})
or for verifying EFBS as discussed in section~\ref{subsec:essential-backward-stablity}.
Note that line \ref{ln:approx-conv} is an approximate check due to
quasi-minimization.

An important detail in Algorithm~\ref{alg:FBSMR} is that we use
$\boldsymbol{M}^{-1}\boldsymbol{b}$ as the initial guess, even though
$\boldsymbol{M}$ in general may be a grossly inaccurate approximation
to $\boldsymbol{A}$. This choice is different from the standard practice
of GMRES \cite{Saad:2003aa}, although it is consistent with the practice
in IR (see e.g., \cite{arioli2009using}). We will give heuristic
and numerical justifications for its use in the context of RP-GMRES
in section~\ref{subsec:Effect-of-initial}. Also note that in line~\ref{ln:Arnoldi},
either MGS or CGS can be used without compromising EFBS due to Theorem~\ref{thm:restarted-rpgmres}.
It is advisable to use MGS in serial but CGS in parallel (see section~\ref{subsec:Comparison-of-CGS}).

In terms of implementation, only $\tilde{\boldsymbol{x}}$ needs to
be stored in higher precision; we use `$\leftleftarrows$' to emphasize
the assignment in higher precision. Among the arithmetic operations,
only the gaxpy operations (including matrix-vector multiplications)
that are boxed are in higher-precision. For example, if the working
precision is double precision, then the preconditioner $\boldsymbol{M}$
would be in single precision, and the gaxpy operations can be in double-double
precision \cite{briggs1998doubledouble}. If the working precision
is single precision, then the preconditioner $\boldsymbol{M}$ may
be in half precision, and the gaxpy operations can be in double precision.
If an operand for high-precision operation is stored in working precision,
we convert the numbers into higher precision in a just-in-time fashion
before each floating-point arithmetic, which is easy to implement
for gaxpy.

Another implementation detail is the computation of the Givens rotations
$\boldsymbol{\Omega}_{i}$ in line~\ref{ln:givens-1}. For complex
matrices, $\boldsymbol{\Omega}_{i}=\begin{bmatrix}c & s\\
-\bar{s} & c
\end{bmatrix}$, where $c\in\mathbb{R}$ and $s\in\mathbb{C}$ such that $\boldsymbol{\Omega}_{i}\begin{bmatrix}h_{k,k}\\
\alpha
\end{bmatrix}=\begin{bmatrix}r\\
0
\end{bmatrix}$ for $r\in\mathbb{C}$. The choice of $c,$ $s$, and $r$ is not
unique. Following \cite{bindel2002computing,basic2002basic}, we choose
$r=\text{sign}(h_{k,k})r_{0}$ with $r_{0}=\sqrt{\left|h_{k,k}\right|^{2}+\alpha^{2}}$
and $\text{sign}(x)=\begin{cases}
x/\vert x\vert & x\neq0\\
1 & x=0
\end{cases}$, and then $c=\vert h_{k,k}\vert/r_{0}$ and $s=\text{sign}(h_{k,k})\alpha/r_{0}$.
$\hat{\boldsymbol{\Omega}}_{i}=\begin{bmatrix}\boldsymbol{I}_{i-1}\\
 & \boldsymbol{\Omega}_{i}\\
 &  & \boldsymbol{I}_{k-i}
\end{bmatrix}$ expands $\boldsymbol{\Omega}_{i}$ to $(k+1)\times(k+1)$ dimensions
at the $k$th step and $\hat{\boldsymbol{Q}}_{k}=\hat{\boldsymbol{\Omega}}_{1}\dots\hat{\boldsymbol{\Omega}}_{k}$
as in Algorithm~\ref{alg:PGMRES}. The residual norm of (\ref{eq:petrov-galerkin})
is estimated as $\left|\boldsymbol{e}_{k+1}^{T}\boldsymbol{g}\right|=\left|\boldsymbol{e}_{k+1}^{T}\hat{\boldsymbol{Q}}_{k}^{H}(\beta\boldsymbol{e}_{1})\right|$,
where $\boldsymbol{g}$ is updated incrementally.

If the approximate inverse is accurate enough or the restart value
is large enough so that the true residual can be reduced before each
restart, then Theorem~\ref{thm:restarted-rpgmres} guarantees convergence
of FBSMR. We will demonstrate numerically that it is indeed the case
even when $\kappa(\boldsymbol{A})$ is close to $1/\varepsilon_{w}$.

\section{\label{sec:numerical-experiments}Numerical Experimentations}

In this section, we report some numerical results using both random
and realistic matrices from PDEs. We implemented FBSMR in MATLAB.
For double-double precision, we implemented the addition and multiplication
operations in MATLAB based on the algorithm described in \cite{hida2007library}
and then converted some computationally intensive parts into C++ using
MATLAB Coder \cite{MATLAB:R2022b}. Although efficiency is one of
our goals, we do not report runtimes in this work since the implementation
is not yet fully optimized. We plan to release the optimized implementation
in C++ as an open-source library and will report it elsewhere. 

\subsection{\label{subsec:Verification-of-EFBS}Verification of EFBS and convergence
criteria}

We first verify our analysis of EFBS and algorithm FBSMR by presenting
some results using ``randsvd'' matrices as described in Example~\ref{exa:random}.
To this end, we focus on the effect of stabilizing forward and backward
errors in FBSMR, so we use RP-GMRES (as in \cite{arioli2009using})
as the baseline for comparison. In all our tests, we used double-precision
LU in MATLAB \cite{MATLAB:R2022b} to compute the factorization. In
this case, FBSMR can be considered as an alternative to EPIR with
backward-error-based convergence criteria, compared to forward-error-based
criteria \cite{demmel2006error}. We generated six matrices using
a combination of two sizes ($n\in\{100,200\}$) and three $\alpha$
values ($\alpha\in\{10,12,14\}$), with $r=1$ for all cases. For
each test case, we first generated a ``randsvd'' $\boldsymbol{A}$
and a random $\boldsymbol{b}$ in double precision, starting with
a preset seed (1) for reproducibility. We then converted $\boldsymbol{A}$
and $\boldsymbol{b}$ into variable-precision arithmetic (VPA) in
MATLAB (\url{https://www.mathworks.com/help/symbolic/vpa.html}) and
then computed $\boldsymbol{x}=\boldsymbol{A}\backslash\boldsymbol{b}$
using VPA. To prevent VPA from converting floating-point numbers into
closest rational numbers and introducing a rounding error of $\varepsilon_{w}$,
we first printed the numbers in $\boldsymbol{A}$ and $\boldsymbol{b}$
into character strings as quadruple-precision numbers with 32-digit
precision in C++ and then converted the strings into VPA. For all
the tests, we set the convergence tolerance ($\tau$) to $10^{-15}$. 

In Table~\ref{tab:Forward-and-backward}, we report the forward and
backward errors in the numerical solutions from RP-GMRES and FBSMR,
respectively. For FBSMR, we give the forward errors in terms of both
the working-precision $\hat{\boldsymbol{x}}$ and the higher-precision
$\tilde{\boldsymbol{x}}$. All the errors are computed in VPA using
the exact $\boldsymbol{x}$ and $\boldsymbol{b}$ as reference solutions.
It can be seen that RP-GMRES could only achieve approximately $\kappa(\boldsymbol{A})\varepsilon_{w}$
as expected. Furthermore, the backward error appeared to be quite
sensitive to a moderate increase in the problem size. In contrast,
FBSMR consistently delivered solutions at the maximal accuracy of
$\varepsilon_{w}$, and the errors were insensitive to problem sizes.
These results verified our analysis of EFBS and FBSMR. As a side product,
this experiment also confirmed our conclusion that randsvd matrices
are intrinsically well-conditioned for random $\boldsymbol{b}$ in
Example~\ref{exa:random}. It is also worth noting that $\tilde{\boldsymbol{x}}$
and $\hat{\boldsymbol{x}}$ have similar accuracy. However, if the
backward errors of FBSMR were computed from $\tilde{\boldsymbol{x}}$
(not shown in Table~\ref{tab:Forward-and-backward} to avoid confusion)
instead of $\hat{\boldsymbol{x}}$, then they would have been only
slightly better than those of RP-GMRES since $\kappa_{\delta\boldsymbol{x}}$
is close to $\kappa(\boldsymbol{A})$ due to Theorem~\ref{thm:conservation-condition-number}.
In other words, the $\mathcal{O}(\varepsilon_{w})$ errors in $\tilde{\boldsymbol{x}}$
are strongly correlated (in terms of $\boldsymbol{A}(\boldsymbol{x}-\tilde{\boldsymbol{x}}$))
but those in $\hat{\boldsymbol{x}}$ are uncorrelated. 

\begin{table}
\caption{\label{tab:Forward-and-backward}Forward and backward errors in solving
systems with ``randsvd'' $\boldsymbol{A}$ and random $\boldsymbol{b}$
using RP-GMRES and FBSMR preconditioned with LU factorization of $\boldsymbol{A}$.}

\centering{}%
\begin{tabular}{cc|cc|ccc}
\hline 
\multicolumn{2}{c|}{Case} & \multicolumn{2}{c|}{RP-GMRES} & \multicolumn{3}{c}{FBSMR}\tabularnewline
$n$ & $\text{\ensuremath{\kappa}}(\boldsymbol{A})$ & $\frac{\Vert\hat{\boldsymbol{x}}-\boldsymbol{x}_{*}\Vert}{\Vert\boldsymbol{x}_{*}\Vert}$ & $\frac{\Vert\boldsymbol{b}-\boldsymbol{A}\hat{\boldsymbol{x}}\Vert}{\Vert\boldsymbol{b}\Vert}$ & $\frac{\Vert\hat{\boldsymbol{x}}-\boldsymbol{x}_{*}\Vert}{\Vert\boldsymbol{x}_{*}\Vert}$ & $\frac{\Vert\tilde{\boldsymbol{x}}-\boldsymbol{x}_{*}\Vert}{\Vert\boldsymbol{x}_{*}\Vert}$ & $\frac{\Vert\boldsymbol{b}-\boldsymbol{A}\tilde{\boldsymbol{x}}\Vert}{\Vert\boldsymbol{b}\Vert}$\tabularnewline
\hline 
\hline 
100 & 1.e10 & 5.97e-8 & 8.96e-8 & 3.92e-17 & 2.73e-22 & 5.01e-22\tabularnewline
200 & 1.e10 & 5.43e-8 & 1.13e-7 & 4.70e-17 & 1.29e-22 & 2.27e-22\tabularnewline
\hline 
100 & 1.e12 & 6.02e-6 & 7.45e-6 & 7.83e-17 & 6.80e-17 & 1.01e-16\tabularnewline
200 & 1.e12 & 9.75e-6 & 1.12e-5 & 8.78e-17 & 7.76e-17 & 3.24e-16\tabularnewline
\hline 
100 & 1.e14 & 6.63e-4 & 6.71e-4 & 4.60e-17 & 6.80e-17 & 5.27e-18 \tabularnewline
200 & 1.e14 & 3.62e-4 & 1.17e-3 & 4.62e-17 & 6.35e-18 & 2.03e-17\tabularnewline
\hline 
\end{tabular}
\end{table}

\subsection{Evaluation with sparse linear systems}

In this test, we use some ``real-life'' systems from the UF Sparse
Matrix Collection \cite{DavisHu11UFSPC}. The first two cases were
used in \cite{arioli2009using}, and the next two were used in \cite{carson2017new}.
We added two additional problems from \cite{DavisHu11UFSPC}, where
\textsf{mplate} is complex non-Hermitian and \textsf{invextr1\_new}
has a user-provided right-hand-side (RHS). For the problem that did
not have the RHS vectors, we used all ones for $\boldsymbol{b}$.
To be as realistic as possible, we used single-precision MUMPS \cite{amestoy2001mumps}
to evaluate the factorization and triangular solves for all the cases.
The first two systems are symmetric and positive definite (SPD), so
we used single-precision Cholesky factorization; for the other systems,
we used single-precision LU with pivoting. Other than these settings,
we treated MUMPS as a black box. Table~\ref{tab:Forward-and-backward-UFSP}
reports the backward errors from RP-GMRES and FBSMR. We omit the forward
errors since small backward errors do imply small forward errors for
unpolluted well-posed problems, as we have proven in section~\ref{subsec:intrinsic-condition-numbers}
and demonstrated in section~\ref{subsec:Verification-of-EFBS}. In
addition, it is impractical to use VPA to evaluate the exact $\boldsymbol{x}$
for larger sparse systems. It can be seen that FBSMR achieved $\varepsilon_{w}$
consistently, but RP-GMRES achieved about $\sqrt{\varepsilon_{w}}$
for five out of six cases. It is worth noting that IR would have converged
only for \textsf{adder\_dcop\_06} and diverged for all the others.

\begin{table}
\caption{\label{tab:Forward-and-backward-UFSP}Backward errors in solving sparse
systems from \cite{DavisHu11UFSPC} using RP-GMRES and FBSMR preconditioned
with single-precision factorization using MUMPS \cite{amestoy2001mumps}.
$\kappa_{1}(\boldsymbol{A})$ were estimated using the \textsf{condest}
function in MATLAB.}

\centering{}\setlength\tabcolsep{2pt}%
\begin{tabular}{cccccc|cc}
\hline 
\multicolumn{6}{c|}{Case} & \multicolumn{1}{c}{RP-GMRES} & \multicolumn{1}{c}{FBSMR}\tabularnewline
\hline 
id & $n$ & $\text{\ensuremath{\kappa}}_{1}(\boldsymbol{A})$ & type & RHS & Application & $\frac{\Vert\boldsymbol{b}-\boldsymbol{A}\hat{\boldsymbol{x}}\Vert}{\Vert\boldsymbol{b}\Vert}$ & $\frac{\Vert\boldsymbol{b}-\boldsymbol{A}\tilde{\boldsymbol{x}}\Vert}{\Vert\boldsymbol{b}\Vert}$\tabularnewline
\hline 
\hline 
s3rmq4m1 & 5489 & 3.1e10 & spd & \textbf{1} & structural & 7.87e-7 & 9.73e-16\tabularnewline
s3dkq4m2 & 90449 & 3.5e11 & spd & \textbf{1} & structural & 3.06e-6 & 9.82e-16\tabularnewline
radfr1 & 1048 & 5.6e10 & unsym & \textbf{1} & chem. eng. & 1.29e-8 & 1.18e-16\tabularnewline
adder\_dcop\_06 & 1813 & 2.1e12 & unsym & \textbf{1} & circuit sim. & 4.13e-10 & 3.69e-16\tabularnewline
mplate & 5962 & 4.8e16 & nonHerm. & \textbf{1} & acoustic  & 4.46e-7 & 7.69e-17\tabularnewline
invextr1\_new & 30412 & 2.8e18 & unsym & given & CFD & 5.94e-7 & 3.47e-16\tabularnewline
\hline 
\end{tabular}
\end{table}

In section~\ref{sec:forward-and-backward-errors}, we asserted that
the main source of error in RP-GMRES was the inconsistency between
the true residual and estimated residual based on PLS (\ref{eq:projected-least-squares})
(the premise for Theorem~\ref{thm:restarted-rpgmres}). To verify
it, we plot the convergence history of the different residuals from
RP-GMRES in Figure~\ref{fig:Convergence-history-of}, along with
the residuals in FBSMR. It can be seen that the residual of PLS in
RP-GMRES can give a false sense of accuracy in all the cases, which
became obvious only after we corrected the residual at the end of
the computation. More importantly, there is a clear correlation between
these inconsistencies and the lack of accuracy of RP-GMRES. This correlation
is indeed a causation relationship, as we have proven in Theorem~\ref{thm:restarted-rpgmres}.

\begin{figure}
\subfloat[s3rmq4m1]{\includegraphics[width=0.32\columnwidth]{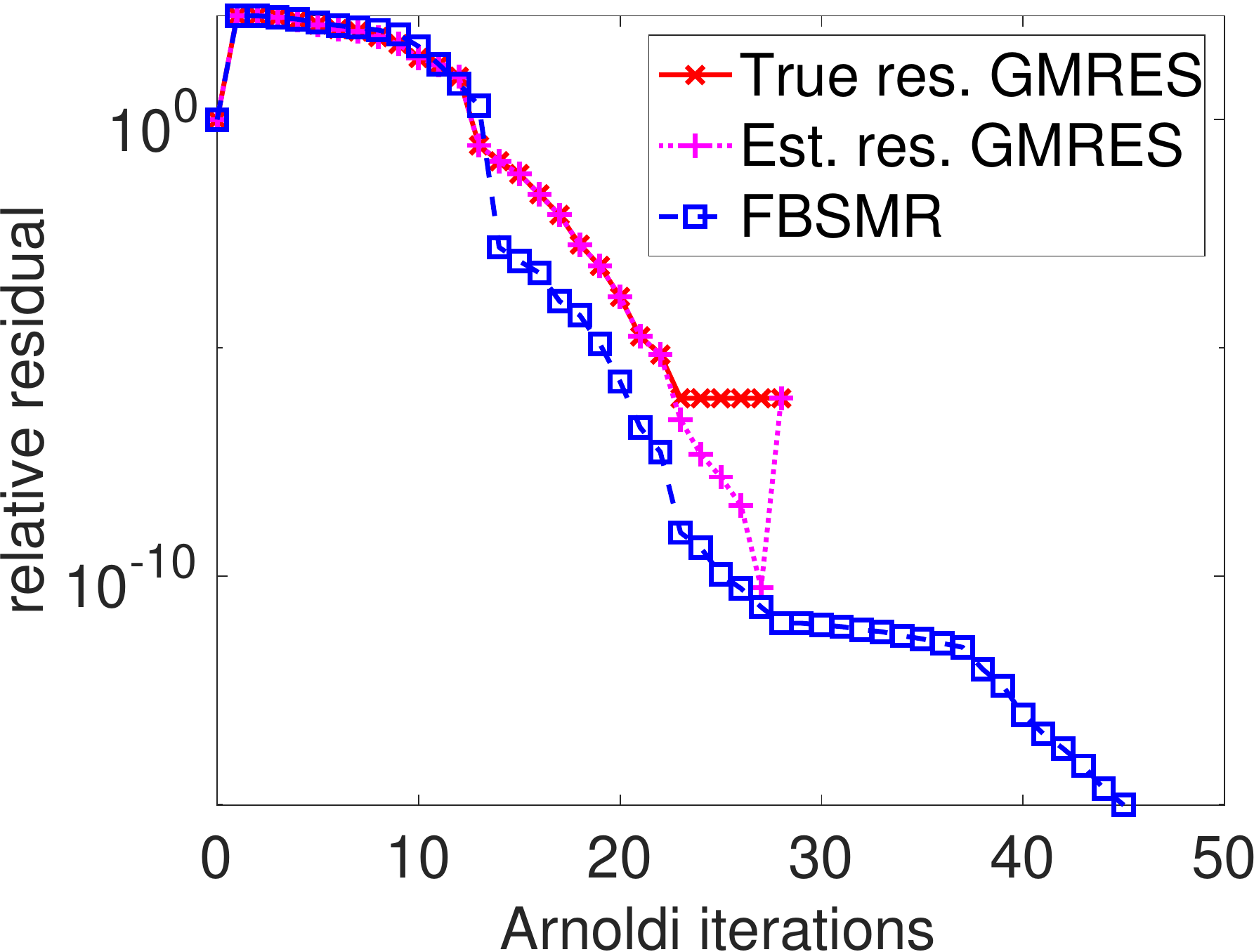}

}\subfloat[s3dkq4m2]{\includegraphics[width=0.32\columnwidth]{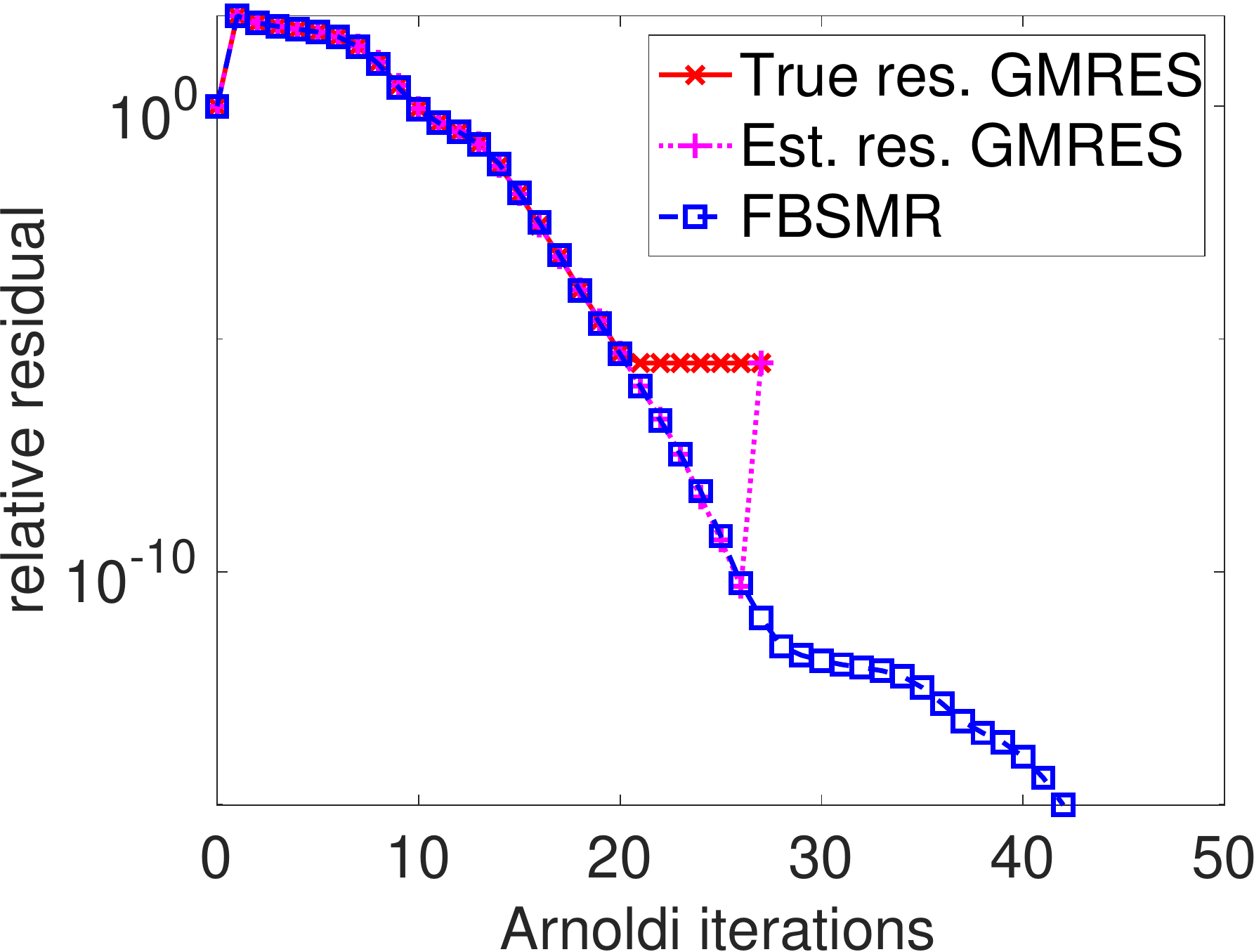}

}\subfloat[radfr1]{\includegraphics[width=0.32\columnwidth]{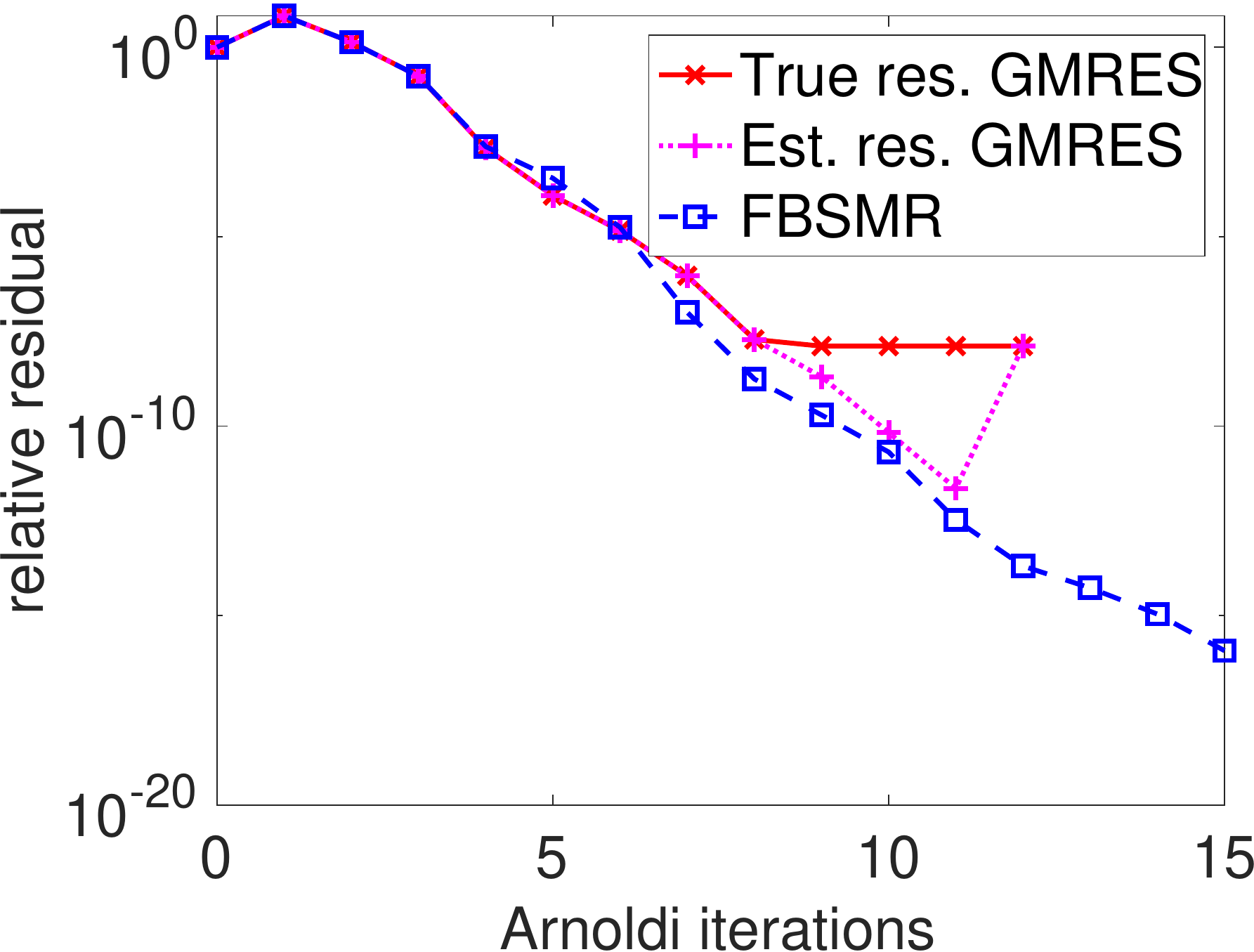}

}

\subfloat[adder\_dcop\_06]{\includegraphics[width=0.32\columnwidth]{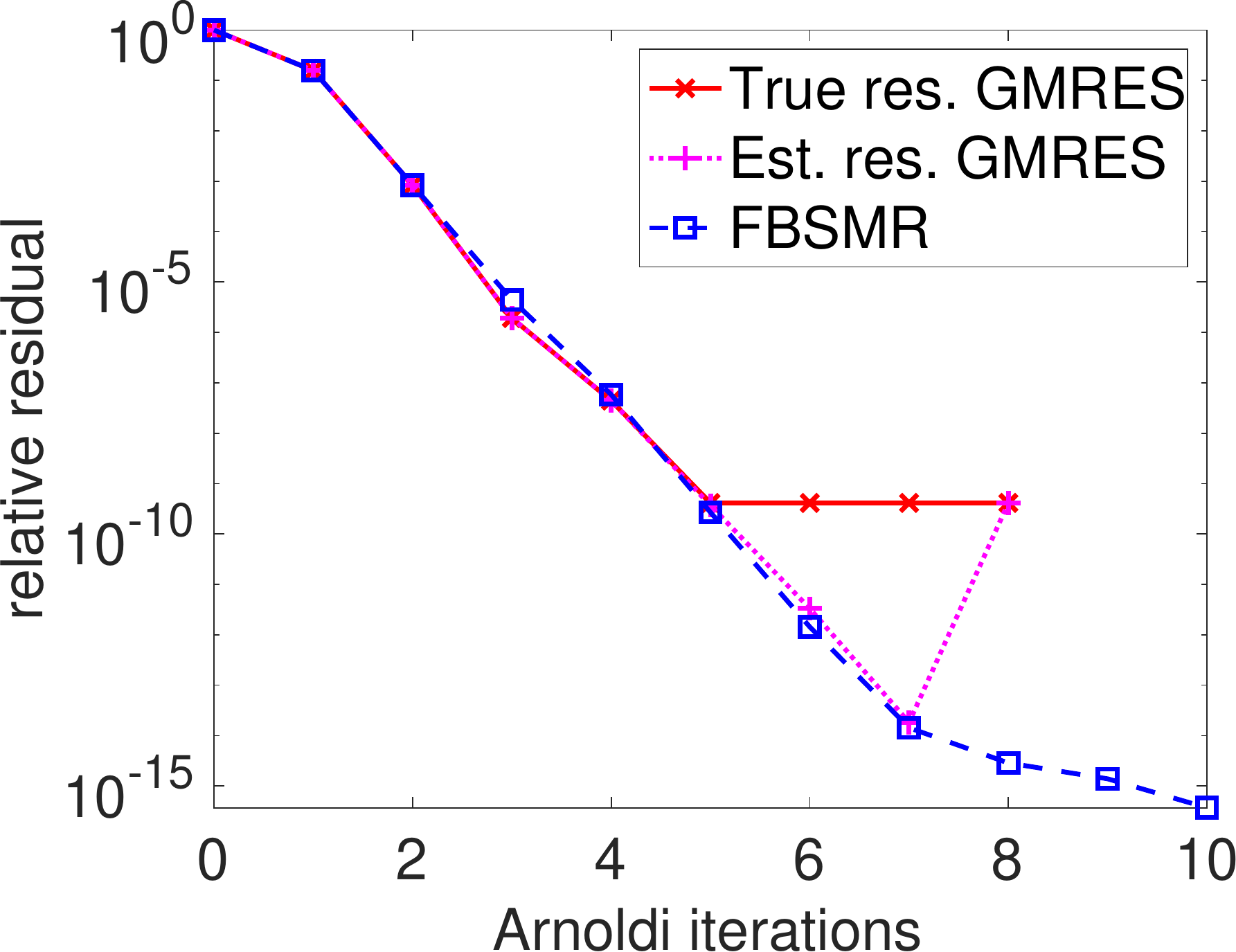}

}\subfloat[mplate]{\includegraphics[width=0.32\columnwidth]{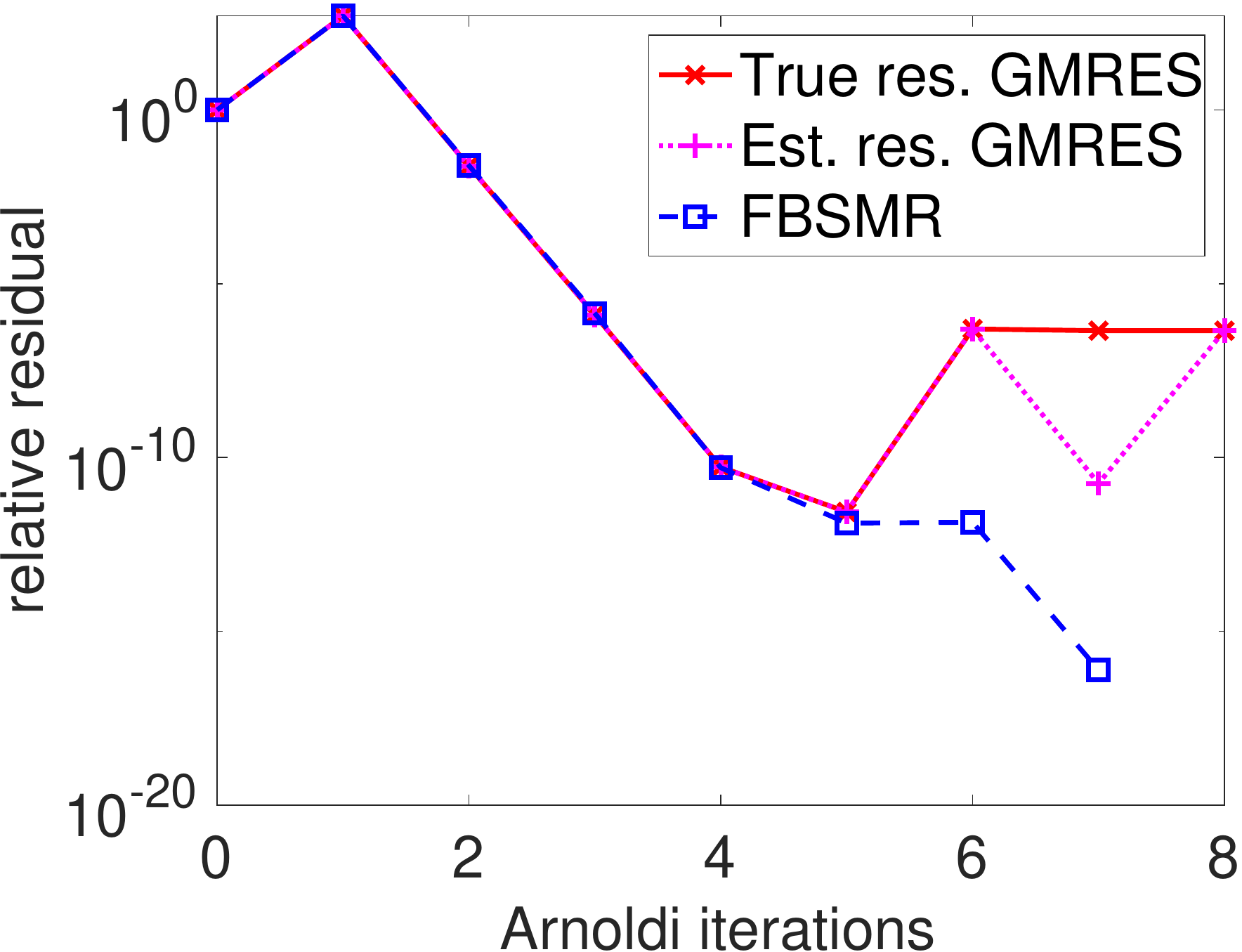}

}\subfloat[invextr1\_new]{\includegraphics[width=0.32\columnwidth]{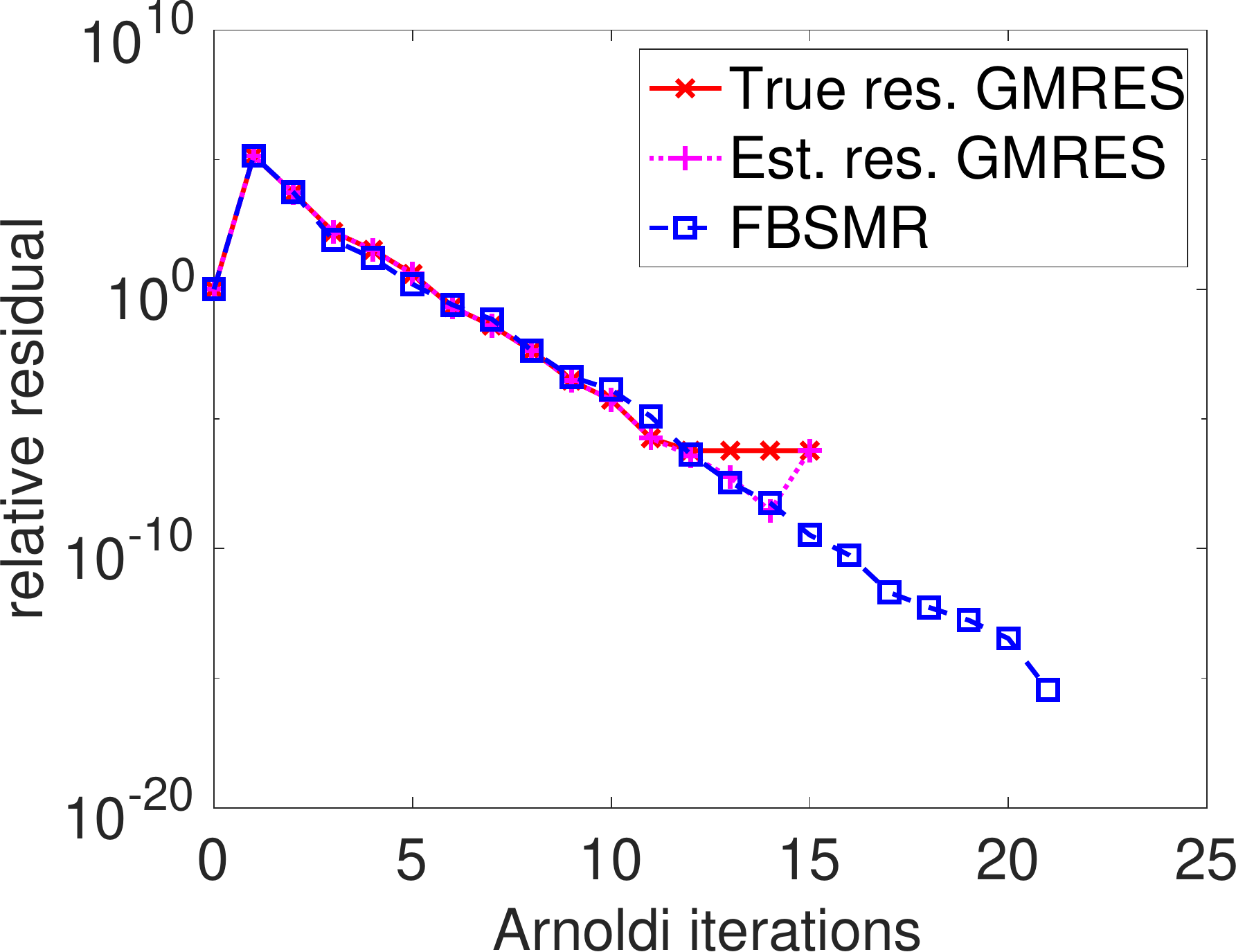}

}

\caption{\label{fig:Convergence-history-of}Convergence history of residuals
for sparse systems in Table~\ref{tab:Forward-and-backward-UFSP}.}
\end{figure}

\subsection{\label{subsec:Effect-of-initial}Effect of initial solutions from
approximate inverses}

In FBSMR, we use $\boldsymbol{x}_{0}=\boldsymbol{M}^{-1}\boldsymbol{b}$
as the initial guess even though $\boldsymbol{M}$ is known to be
an inaccurate factorization. Inaccurate factorizations often lead
to initial spikes of the residuals, as we have seen in Figure~\ref{fig:Convergence-history-of}.
Since the initial spikes may increase the number of iterations, one
might question the wisdom of using $\boldsymbol{M}^{-1}\boldsymbol{b}$
as the initial solution. We conjecture that this spike only indicates
that the singular vectors corresponding to the extreme singular values
of $\boldsymbol{A}$ are not well resolved by the inaccurate factorization.
The approximate factorization can still resolve some (if not all)
singular vectors corresponding to the interior singular values, so
using $\boldsymbol{M}^{-1}\boldsymbol{b}$ as the initial solution
can eliminate some modes corresponding to the interior singular values
from the residual. Hence, we can expect FBSMR to solve the linear
system with lower-dimensional Krylov subspaces, and in turn allowing
a small restart value. A rigorous analysis seems to be challenging
and is beyond the scope of this work. We present some numerical evidence
in Figure~\ref{fig:Convergence-initial} for the cases in Table~\ref{tab:Forward-and-backward-UFSP}.
It can be seen that using $\boldsymbol{x}_{0}=\boldsymbol{0}$ would
make FBSMR converge more slowly in all cases except for \textsf{s3rmq4m1}.
More importantly, with $\boldsymbol{x}_{0}=\boldsymbol{0}$, FBSMR
could not converge for \textsf{mplate} and \textsf{invextr1\_new}
with restart=30. Using $\boldsymbol{M}^{-1}\boldsymbol{b}$ as $\boldsymbol{x}_{0}$
allowed FBSMR to converge even with a small restart value, even though
$\boldsymbol{M}$ may be grossly inaccurate.

\begin{figure}
\subfloat[s3rmq4m1]{\includegraphics[width=0.32\columnwidth]{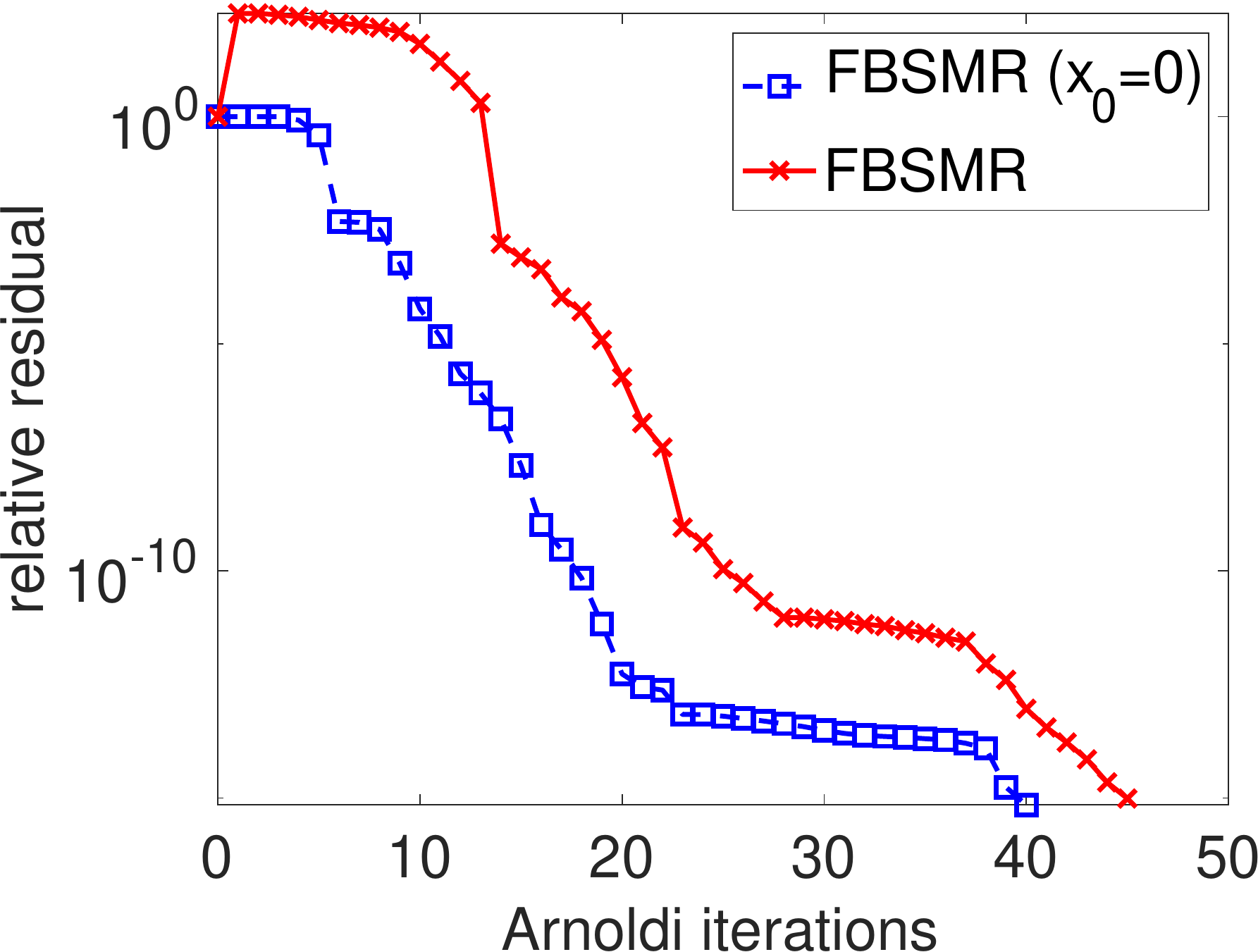}

}\subfloat[s3dkq4m2]{\includegraphics[width=0.32\columnwidth]{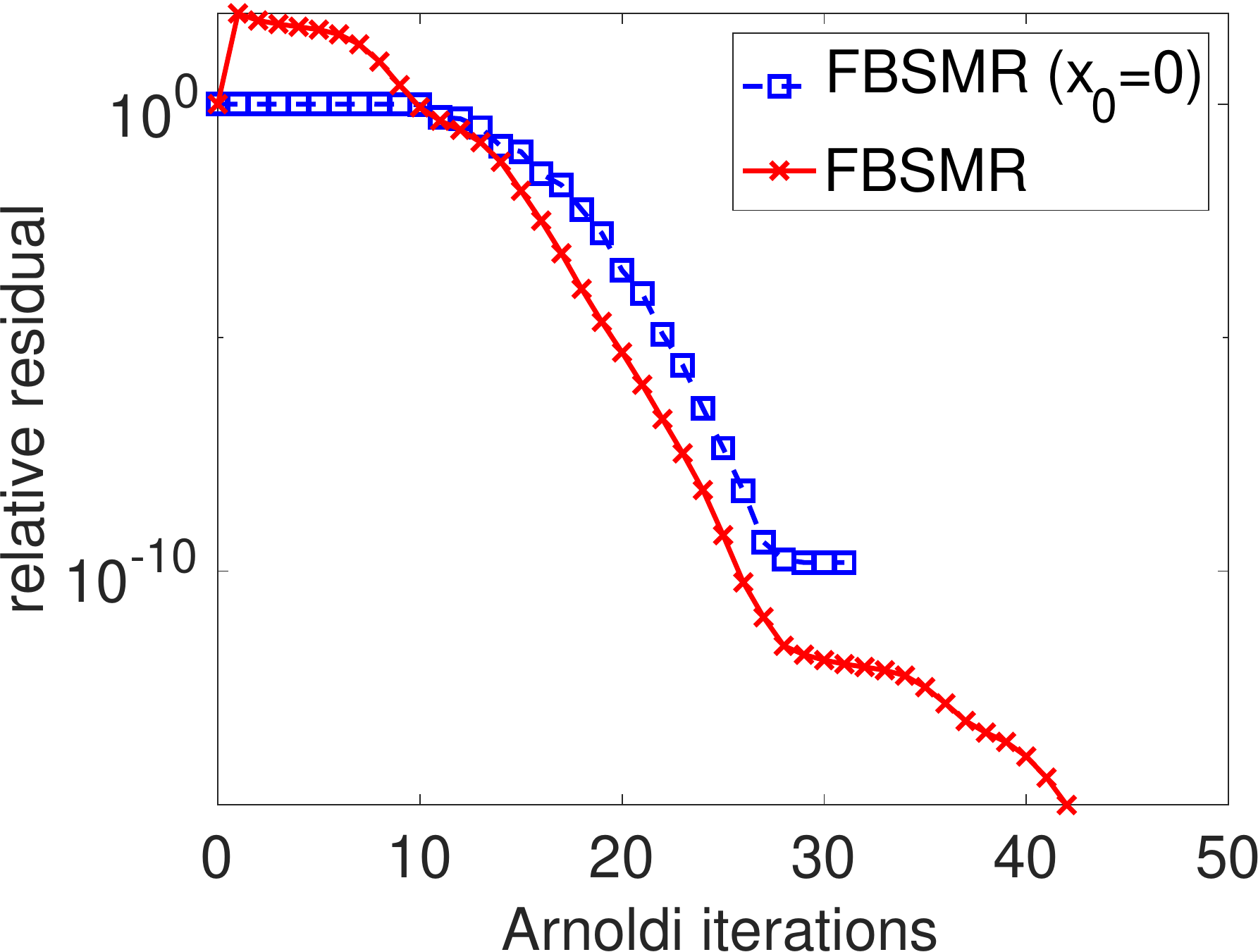}

}\subfloat[radfr1]{\includegraphics[width=0.32\columnwidth]{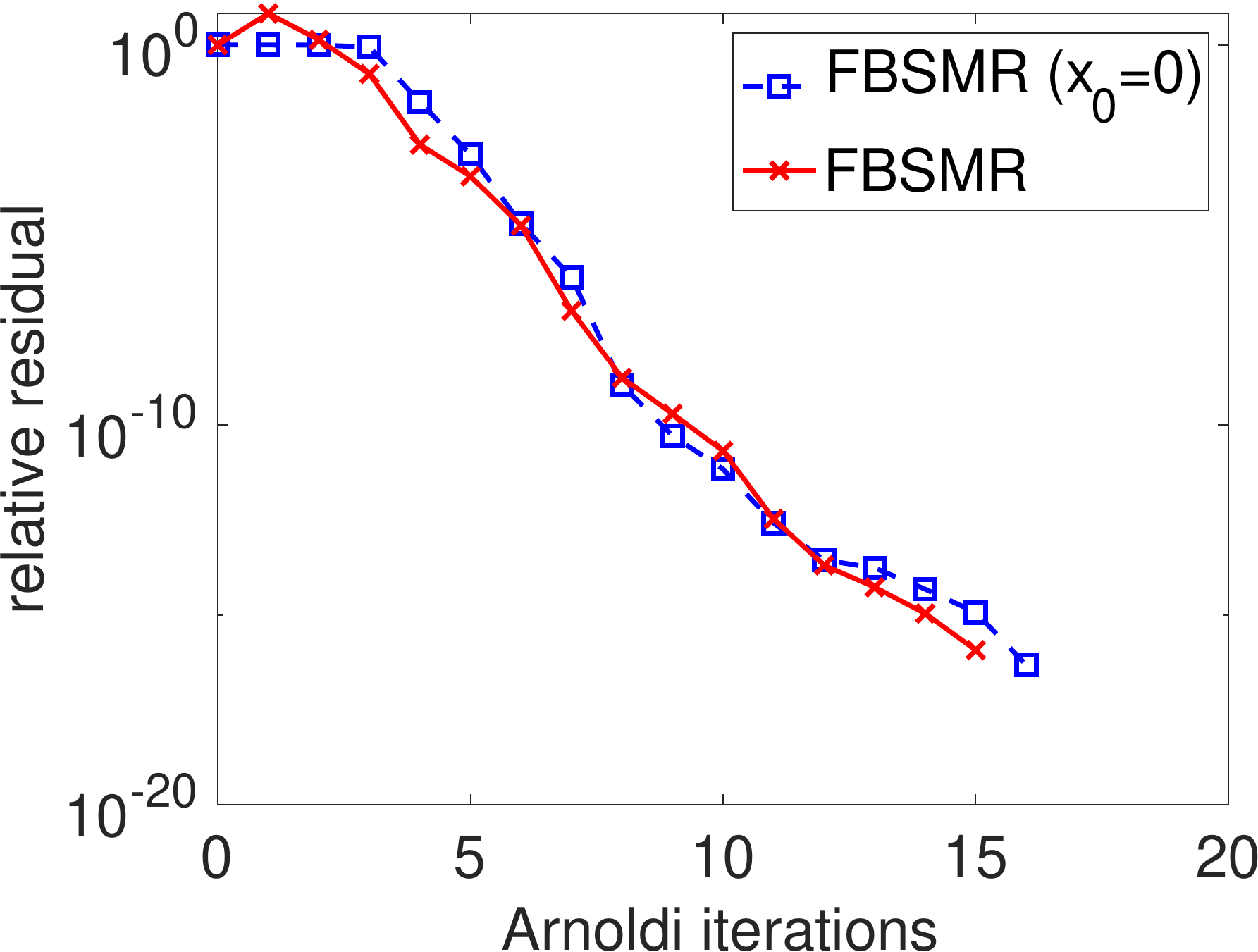}

}

\subfloat[adder\_dcop\_06]{\includegraphics[width=0.32\columnwidth]{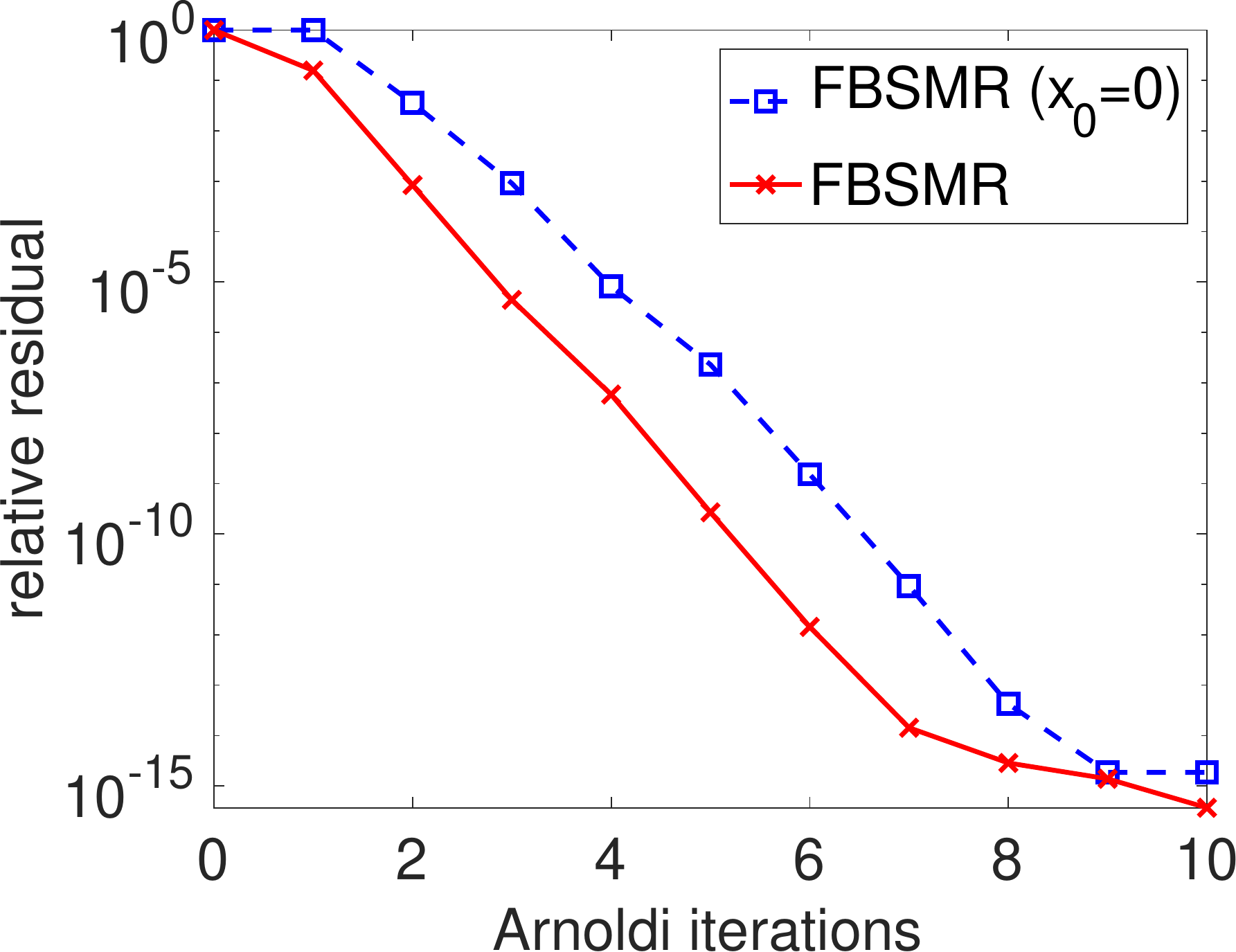}

}\subfloat[mplate]{\includegraphics[width=0.32\columnwidth]{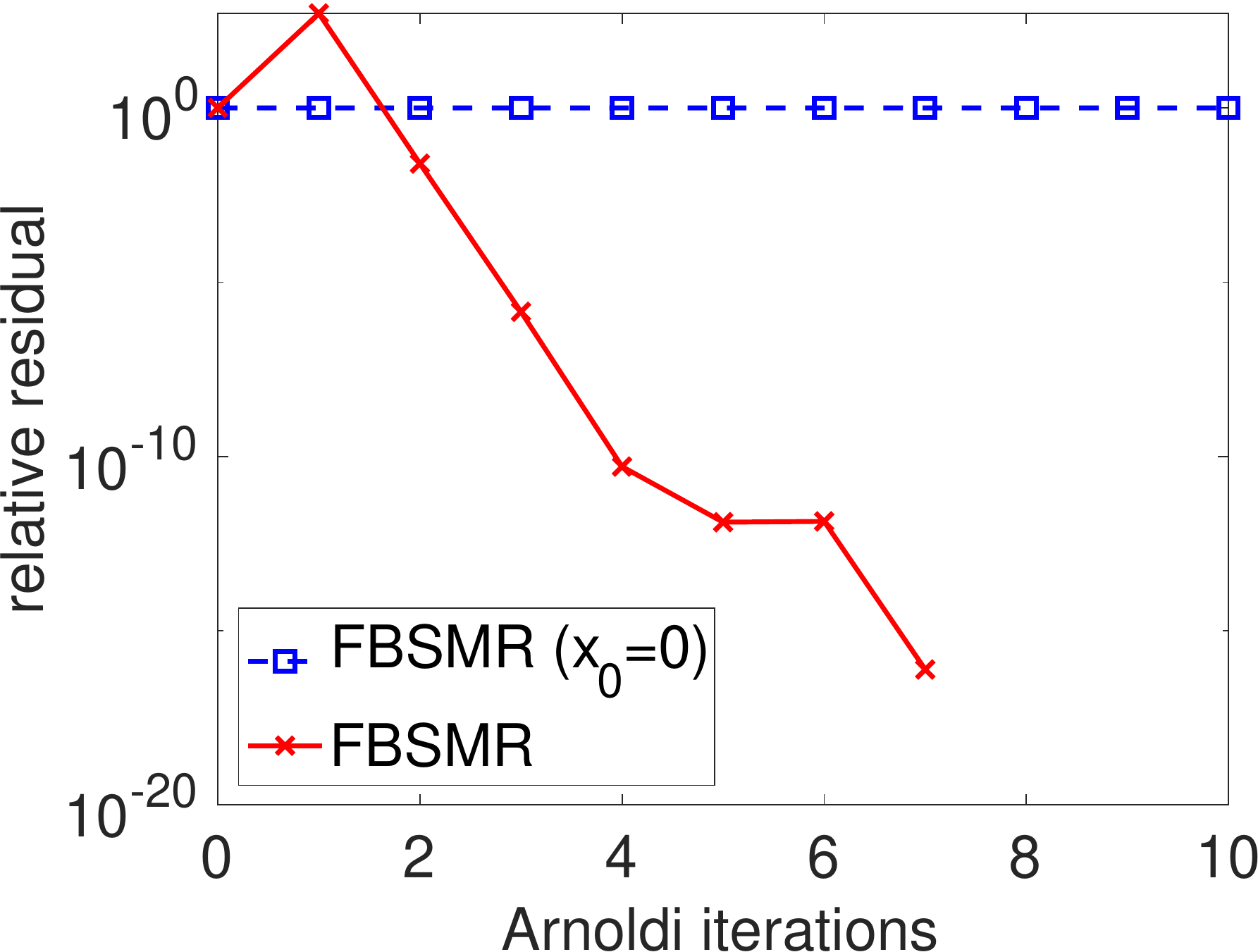}

}\subfloat[invextr1\_new]{\includegraphics[width=0.32\columnwidth]{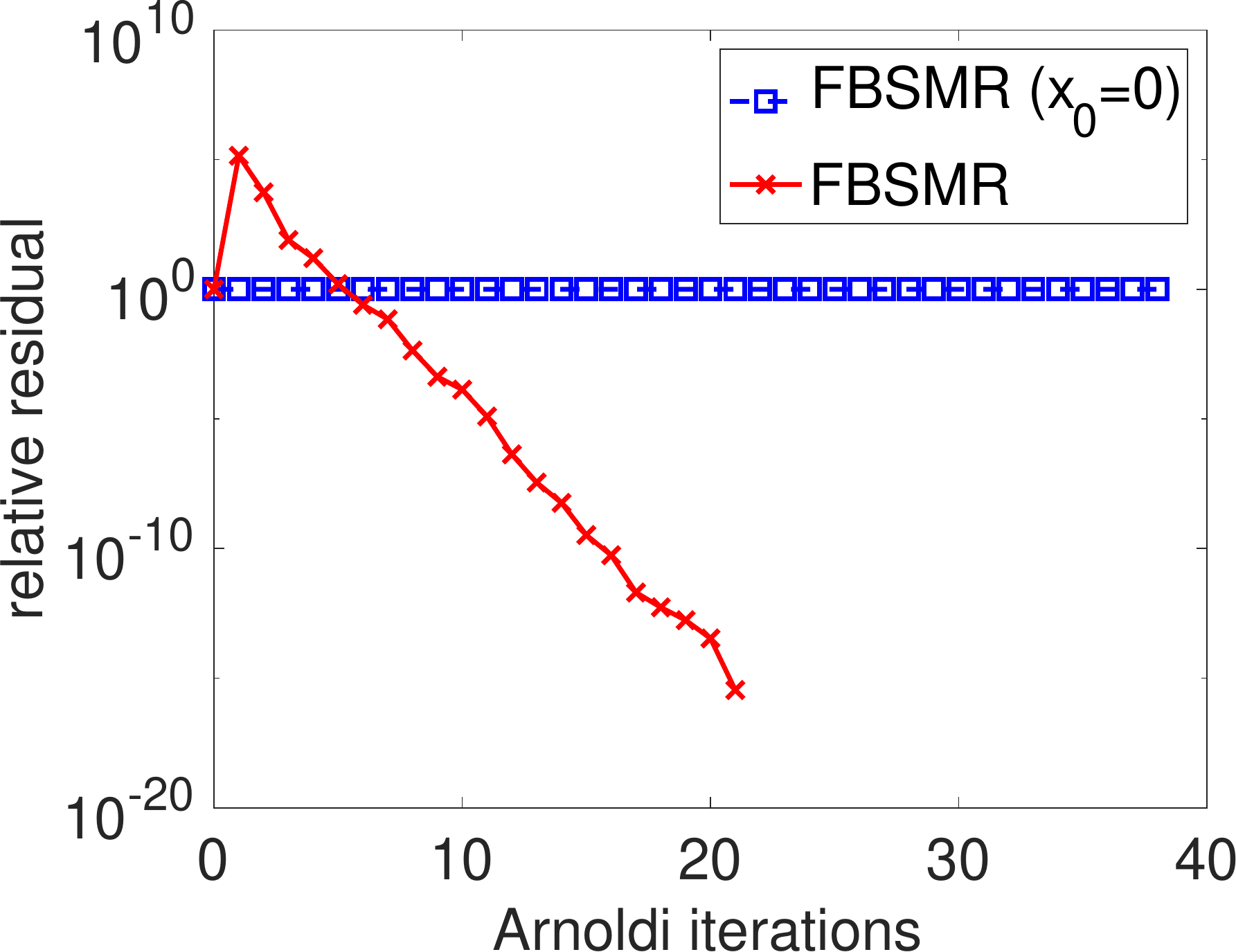}

}

\caption{\label{fig:Convergence-initial}Comparison of $\boldsymbol{x}_{0}=\boldsymbol{M}^{-1}\boldsymbol{b}$
vs. $\boldsymbol{x}_{0}=\boldsymbol{0}$ for sparse systems in Table~\ref{tab:Forward-and-backward-UFSP}.
With $\boldsymbol{x}_{0}=\boldsymbol{0}$, FBSMR did not converge
for \textsf{mplate} and \textsf{invextr1\_new} after 500 iterations.
We cropped the plots to show only the first few iterations.}
\end{figure}

\subsection{\label{subsec:Comparison-of-CGS}Comparison of CGS versus MGS}

In section~\ref{sec:forward-and-backward-errors}, we used the argument
of quasi-minimization to support the use of CGS in place of MGS in
FBSMR. To verify this claim, we compare the convergence history of
FBSMR using MGS versus CGS in Figure~\ref{fig:cgs-vs-mgs} for the
sparse systems in Table~\ref{tab:Forward-and-backward-UFSP}. It
can be seen that FBSMR converged to machine precision with either
MGS or CGS. However, for \textsf{radfr1} and \textsf{adder\_dcop\_06},
FBSMR with MGS used one fewer iteration than with CGS due to a slightly
more accurate estimation of the residual. Hence, we recommend using
MGS in serial but using CGS in parallel to reduce communication overhead
in FBSMR. This practice is common for GMRES, so it would be relatively
easy to implement FBSMR based on existing GMRES implementations.

\begin{figure}
\subfloat[s3rmq4m1]{\includegraphics[width=0.32\columnwidth]{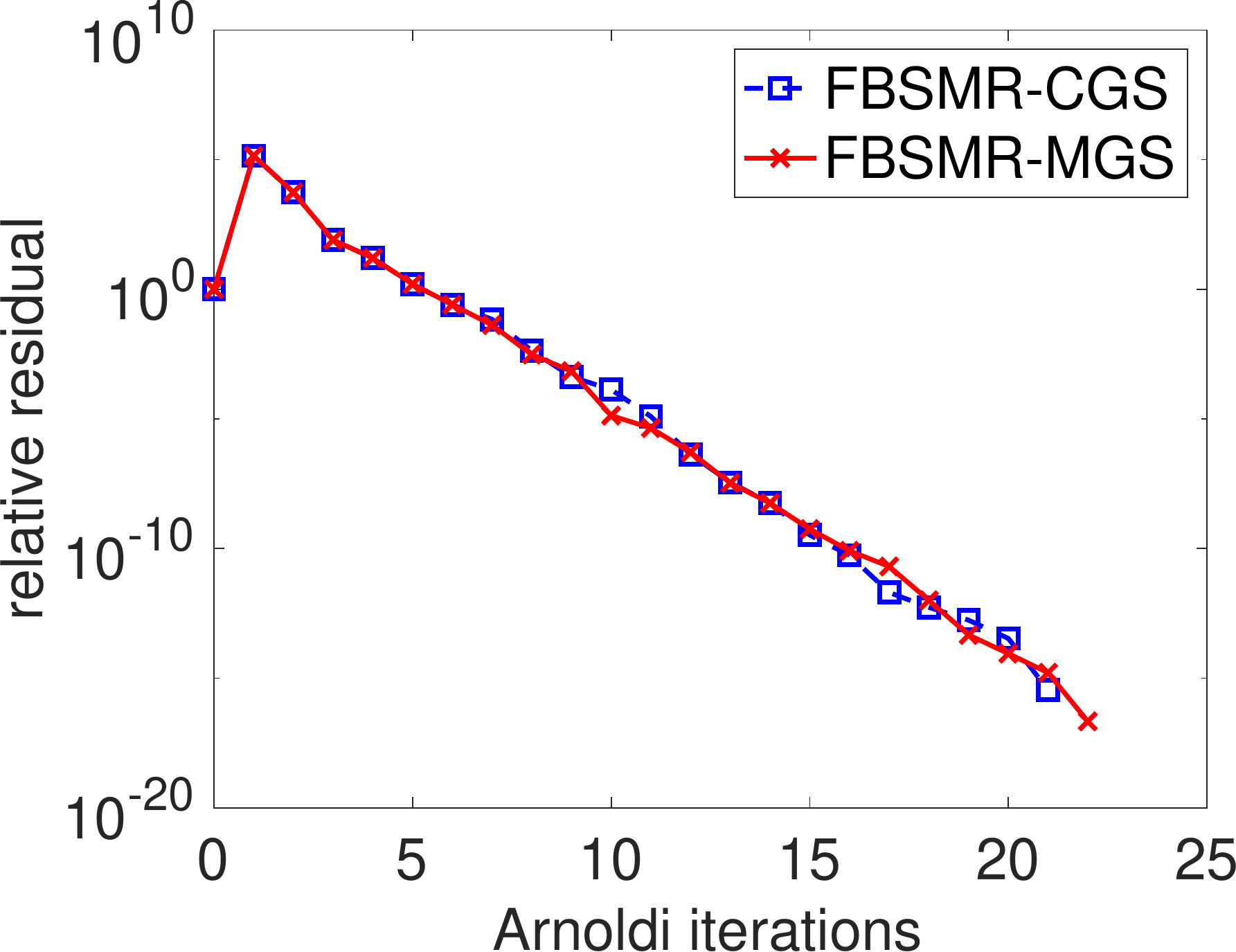}

}\subfloat[s3dkq4m2]{\includegraphics[width=0.32\columnwidth]{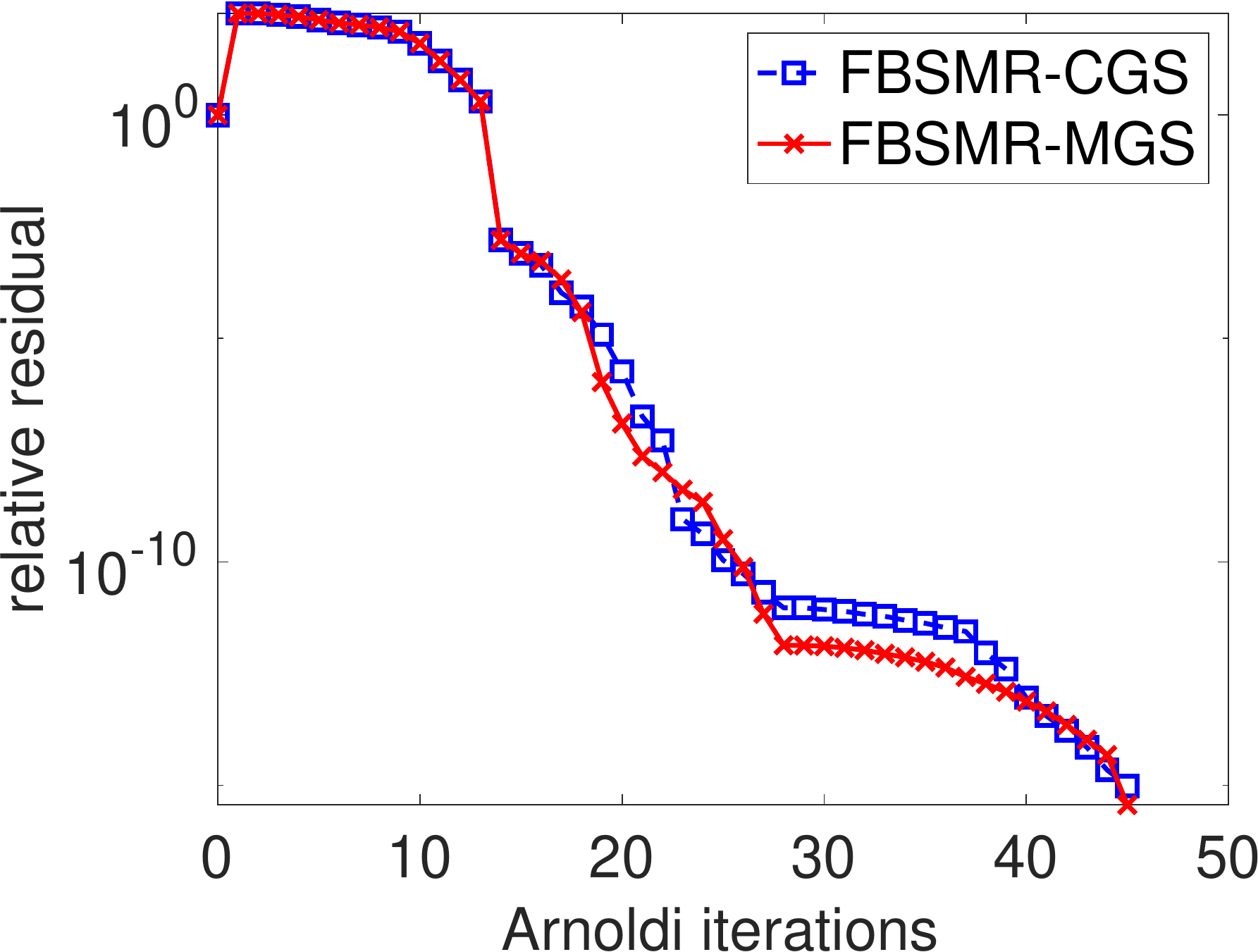}

}\subfloat[radfr1]{\includegraphics[width=0.32\columnwidth]{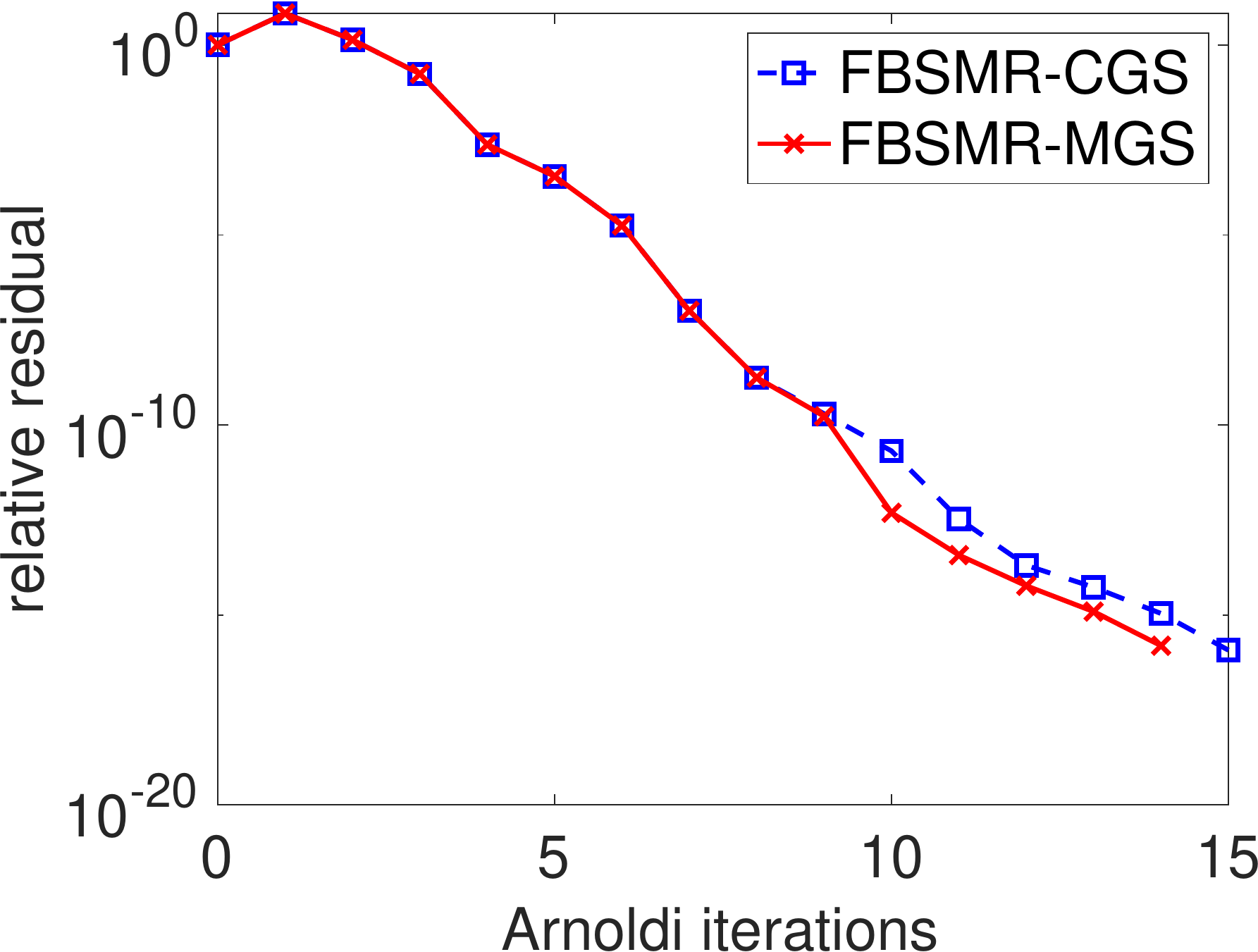}

}

\subfloat[adder\_dcop\_06]{\includegraphics[width=0.32\columnwidth]{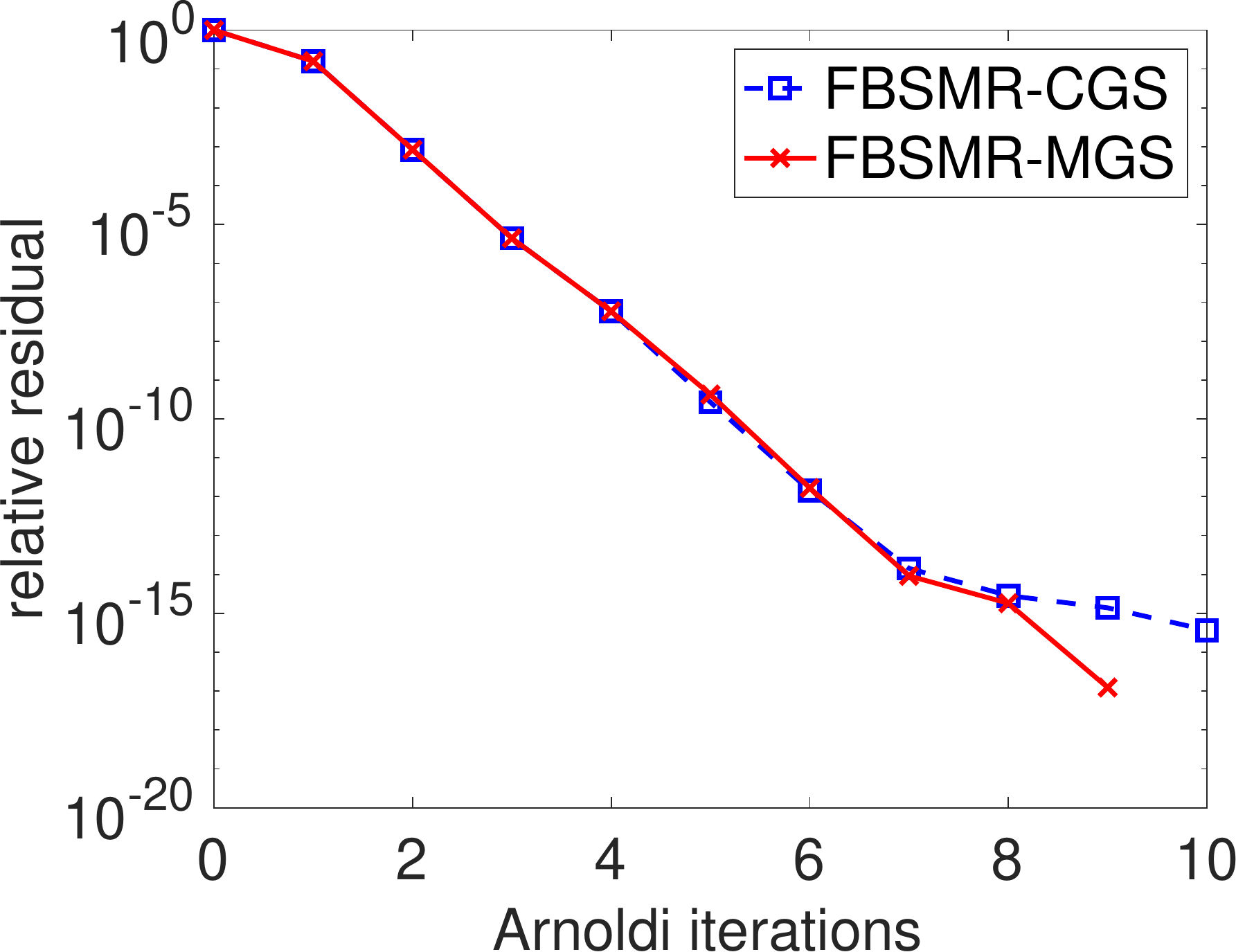}

}\subfloat[mplate]{\includegraphics[width=0.32\columnwidth]{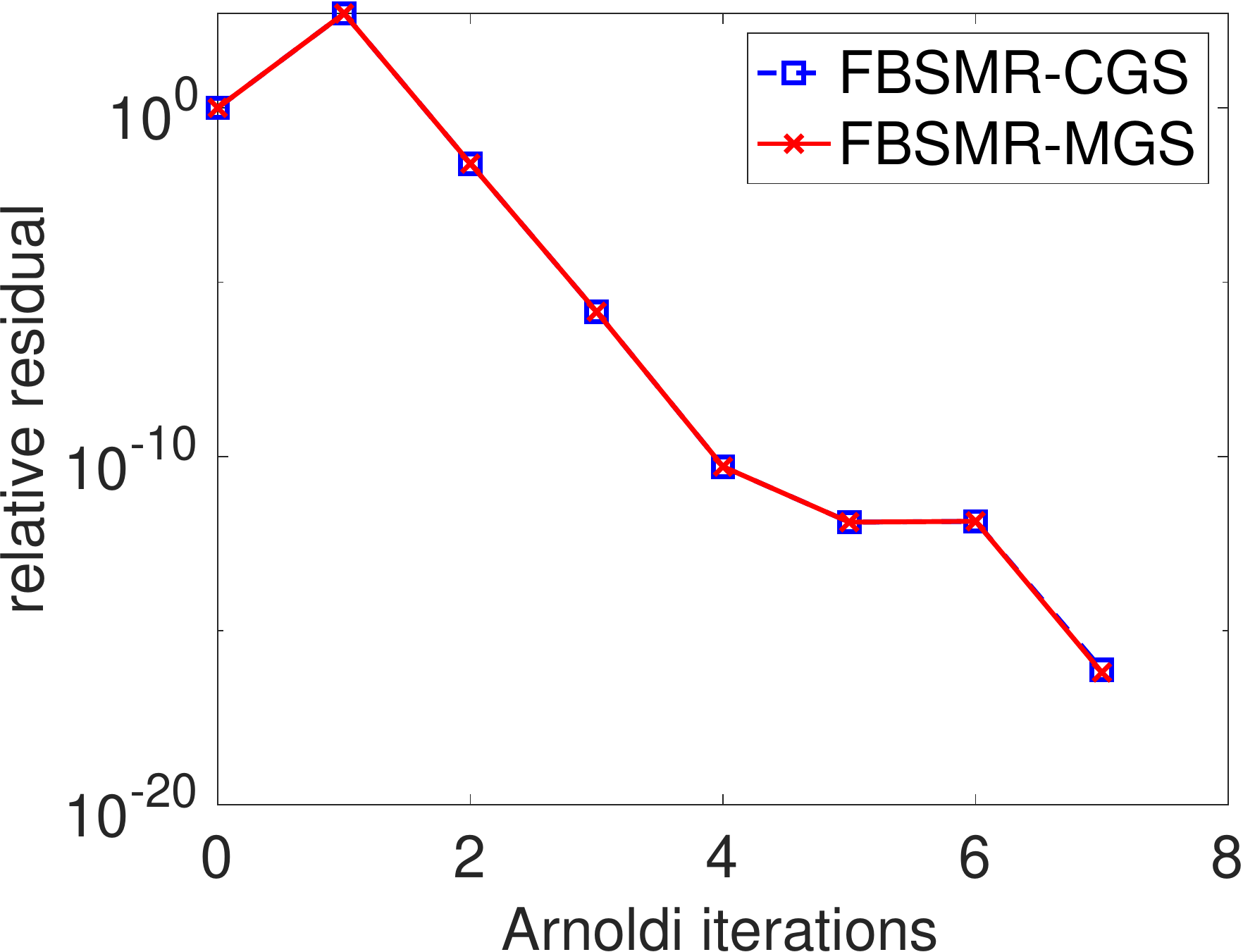}

}\subfloat[invextr1\_new]{\includegraphics[width=0.32\columnwidth]{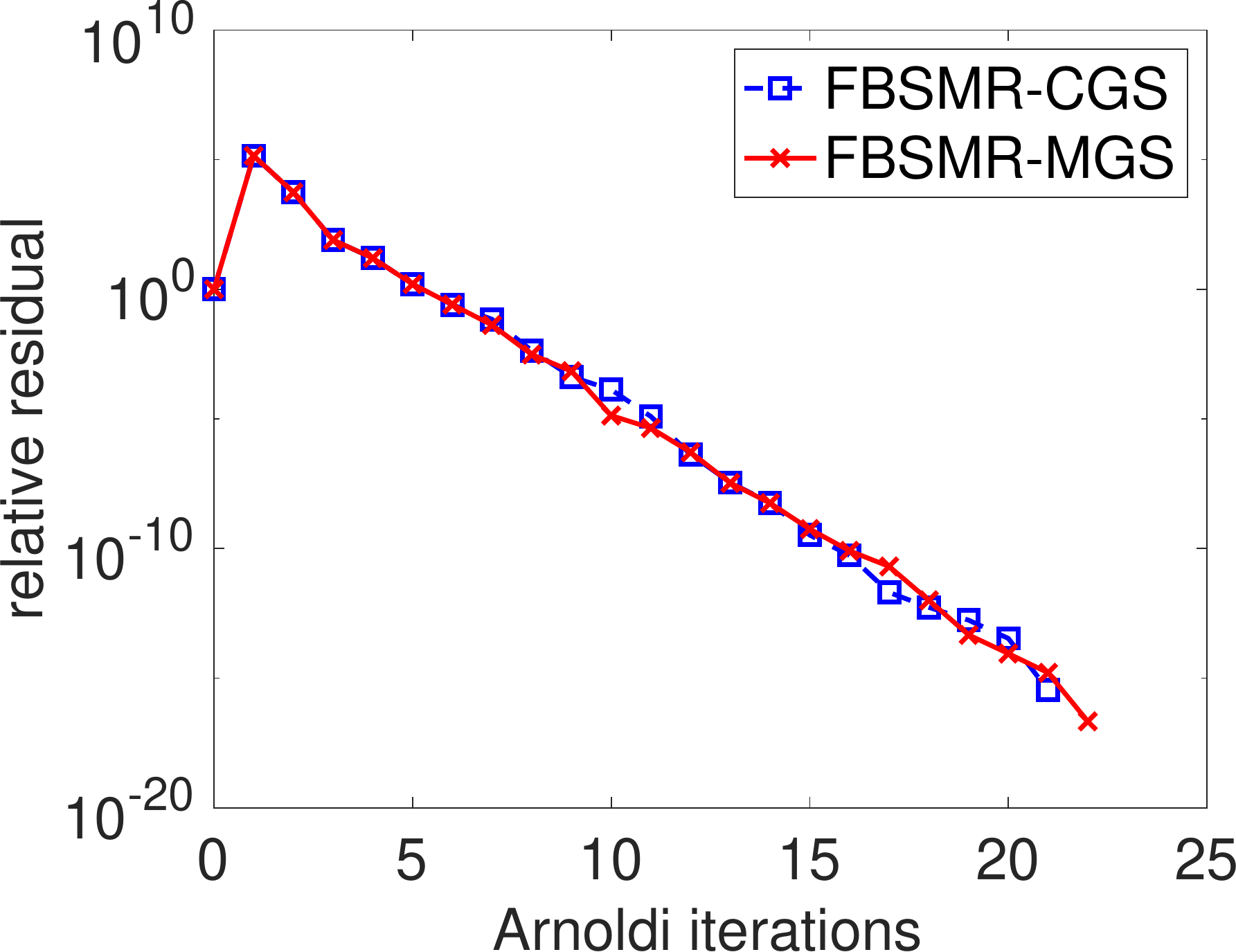}

}

\caption{\label{fig:cgs-vs-mgs}Comparison of CGS vs. MGS for sparse systems
in Table~\ref{tab:Forward-and-backward-UFSP}.}
\end{figure}

\section{\label{sec:Conclusions}Conclusions and Discussions}

In this work, we introduced an iterative method called \emph{FBSMR}
by leveraging quasi-minimization in RP-GMERS for well-posed problems
with ill-conditioned coefficient matrices. We showed that FBSMR could
overcome ill-conditioned coefficient matrices and deliver optimal
forward and backward errors at $\mathcal{O}(\varepsilon_{w})$, assuming
the input matrix is unpolluted in the sense that $\Vert\delta\boldsymbol{A}\boldsymbol{x}\Vert/\Vert\boldsymbol{b}\Vert=\mathcal{O}(\varepsilon_{w})$
for the rounding errors $\delta\boldsymbol{A}$ in $\boldsymbol{A}$.
In contrast, a traditional backward stable method can only guarantee
$\mathcal{O}(\kappa(\boldsymbol{A})\varepsilon_{w})$ in forward errors.
FBSMR achieves this optimal accuracy by meeting a stronger stability
requirement, namely EFBS, using higher precision for a small portion
of the computation. Since the most expensive computations (namely
the preconditioner) can be performed in lower precision, FBSMR can
also improve efficiency over direct methods with or without IR, while
being more accurate than other iterative methods. In this work, we
presented experimental results using LU or Cholesky factorizations
as preconditioners in FBSMR. We have done some preliminary studies
with hybrid incomplete factorization \cite{chen2022hifir} in low
precision and also had great success. We will report more extensive
numerical experimentations, including performance results, when we
release FBSMR as an open-source library in the future.

The concepts developed in this work lead to some interesting new opportunities
in improving numerical methods in several aspects. First, the application
developers can take advantage of EFBS algorithms by devising methods
where the rounding (and discretization) errors in the matrices are
strongly or weakly correlated so that $\Vert\delta\boldsymbol{A}\boldsymbol{x}\Vert/\Vert\boldsymbol{b}\Vert=\mathcal{O}(\varepsilon_{w})$
or $\Vert\delta\boldsymbol{A}\boldsymbol{x}\Vert/\Vert\boldsymbol{b}\Vert=\mathcal{O}(\kappa^{p}(\boldsymbol{A})\varepsilon_{w})$
for some $p<1$ (such as $p=1/2$). Such optimization can significantly
benefit high-order discretization methods, for which the discretization
errors are often so small on finer meshes that rounding errors become
the dominant factor. Second, we can further refine the stopping criteria
in FBSMR for problems with weakly correlated rounding errors, so that
it can stop earlier at $\mathcal{O}(\kappa^{p}(\boldsymbol{A})\varepsilon_{w})$
instead of $\mathcal{O}(\varepsilon_{w})$ based on the applications.
Third, we used a heuristic argument with numerical experimentation
to justify the use of an inaccurate approximate inverse to compute
the initial guess in section~\ref{subsec:Effect-of-initial}; it
would be useful to derive a more rigorous analysis to determine when
it is more efficient to use $\boldsymbol{M}^{-1}\boldsymbol{b}$ as
the initial guess, especially when $\boldsymbol{M}$ is computed only
in $\sqrt[4]{\varepsilon_{w}}$ precision or with incomplete factorization.
Finally, it is possible extend the strategy of FBS to other numerical
problems, such as (rank-deficient) least squares problems and nonlinear
equations, so that more methods can achieve EFBS while taking advantage
of lower-precision preconditioners for improved efficiency. We plan
to explore these research directions in the future. 

\section*{Acknowledgments}

The author would like to thank Dr. Qiao Chen for some helpful discussions
on mixed-precision computations in industrial applications.

\bibliographystyle{abbrv}
\bibliography{krylov}

\end{document}